\documentclass[10pt,a4paper]{article}         

\usepackage{color}
\usepackage{a4,amsmath,amssymb,amscd,latexsym,amsthm,amsbsy} 
\usepackage{bbm} 
\usepackage{epsfig,color} 
\usepackage{amsfonts}
\usepackage{mathrsfs}
\usepackage{mathtools}
\usepackage{overpic}
\usepackage{subcaption}
\usepackage{paralist}
\usepackage{graphicx}
\usepackage{trfsigns}
\usepackage{url}
\usepackage{stmaryrd}
\usepackage{xfrac} 
\usepackage{hyperref}
\usepackage{caption}

\usepackage[]{algorithm2e}
\usepackage{algorithmicx}

\usepackage{pgfplots}				
\usepackage{varwidth}

\usepackage{float}

\newtheoremstyle{mystyle}
{3pt}
{3pt}
{\itshape}
{}
{\bold}
{.}
{.5em}
{}

\newtheorem{definition}{Definition}[section]
\newtheorem{theorem}[definition]{Theorem}
\newtheorem*{theorem*}{Theorem}
\newtheorem{remark}[definition]{Remark}

\newtheorem{lemma}[definition]{Lemma}
\newtheorem{deflem}[definition]{Definition and Lemma}

\numberwithin{equation}{section}

\newcommand{\RR}{\mathbb{R}}
\newcommand{\NN}{\mathbb{N}}

\newcommand{\grad}{\operatorname{grad}}

\renewcommand{\d}{\mathrm{d}}
\newcommand{\dal}{\mathrm{d}_{\alpha}}

\newcommand{\poly}{\mathrm{Poly}}
\newcommand{\paths}{\mathrm{Paths}}
\newcommand{\pathsx}{\mathrm{Paths}_x}

\newcommand{\ra}{\rightarrow}

\newcommand{\Cinf}{C^{\infty}}

\newcommand{\abs}[1]{\lvert #1 \rvert}

\usepackage{hyperref}
\newcommand{\footremember}[2]{%
	\footnote{#2}
	\newcounter{#1}
	\setcounter{#1}{\value{footnote}}%
}

\title{\bf Towards optimization techniques on diffeological spaces by generalizing Riemannian concepts} 

\author{Nico Goldammer\footremember{a}{Helmut-Schmidt-University / University of the Federal Armed Forces Hamburg, Holstenhofweg~85, 22043 Hamburg, Germany; \texttt{goldammer@hsu-hh.de}} and Kathrin Welker\footremember{b}{Helmut-Schmidt-University / University of the Federal Armed Forces Hamburg, Holstenhofweg~85, 22043 Hamburg, Germany; \texttt{welker@hsu-hh.de}} }
\date{}

\begin{document}
\maketitle

\begin{abstract}
	Diffeological spaces firstly introduced by J.M.~Souriau in the 1980s are a natural generalization of smooth manifolds. However, optimization techniques are only known on manifolds so far. 
	Generalizing these techniques to diffeological spaces is very challenging because of several reasons. One of the main reasons is that there are various definitions of tangent spaces which do not coincide. Additionally, one needs to deal with a generalization of a Riemannian space in order to define gradients which are indispensable for optimization methods.
	This paper is devoted to an optimization technique on diffeological spaces.
	Thus, one main aim of this paper is a suitable definition of a tangent space in view to optimization methods. Based on this definition, we present a diffeological Riemannian space and a diffeological gradient, which we need to formulate an optimization algorithm on diffeological spaces. 
	Moreover, in order to be able to update the iterates in an optimization algorithm on diffeological spaces, we present a diffeological retraction and the Levi-Civita connection on diffeological spaces.
	We give examples for the novel objects and apply the presented diffeological algorithm to an optimization problem.
\end{abstract}

\paragraph{Key words.}
Diffeological space, manifold, smooth space, Riemannian metric, tangent space, steepest descent method

\paragraph{AMS subject classifications.} 
49Q10, 53C15, 57P99,54C56, 57R35, 65K05

%
%
%
%
%
%
%


\section{Introduction}
\label{introduction}
	
	The goal of all optimization processes is either to minimize effort or to maximize benefit.
	Optimization techniques are used for example in logistics, control engineering, economics and the finance market.
	Furthermore, the use of optimization in machine learning becomes of greater interest, see, e.g., \cite{SraSuvritHosseiniReshad}.
	There are also several theoretical problems that we can formulate as minimization problems on smooth manifolds like the singular value decomposition in the linear algebra or finding geodesics on shapes (cf.~\cite{TR_MethodsOnR-Manifolds,Knut_Trumpf,WirthRumpf}).
	Theoretical and application problems are highly influenced by each other, e.g., application problems in machine learning lead to new theoretical approaches in the theory \cite{SraSuvritHosseiniReshad}.
	
	Often, optimization problems are posed on Euclidean spaces 
	 \cite{ChallalSamia}
	but there are also several problems that are not posed on an Euclidean space, e.g., the maximization of the Rayleigh quotient on the sphere to find the largest eigenvalue of a symmetric matrix \cite{TrefethenLloydBau}.
	In many circumstances, optimization problems share the same structure, the structure of a smooth Riemannian manifold.
	In the finite-dimensional case, optimization problems on Euclidean spaces can locally be thought of problems on smooth Riemannian manifolds because smooth Riemannian manifolds are isometrically embedded in some Euclidean space \cite{Spivak}.
	However, in this case, it is possible that the dimension of the manifold is smaller than the dimension of the Euclidean space.
	Optimization techniques on manifolds can exploit the underlying manifold structure.
	Thus, many traditional optimization methods on Euclidean spaces have been extended to smooth Riemannian manifold. 
	For example, the steepest decent method, conjugate gradient method and Newton method are formulated on Riemannian manifolds in  \cite{AbsilMahonySepulchre,Smith}. 
	Advanced techniques like BFGS quasi-Newton, Fletcher–Reeves nonlinear conjugate gradient iterations or Dai–Yuan-type Riemannian conjugate gradient method
	were fist consider by Gabay \cite{Gabay} and Udriste \cite{Udriste} and have been extended in \cite{BakerAbsilGallivan,RingWirth,Sato,Yang}.
	
	There are also various disciplines which study infinite-dimensional optimization problems. 
	As examples, we mention optimal control \cite{Fattorini} calculus of variations \cite{Cassel} or deterministic and stochastic shape optimization  \cite{GeiersbachLoayzaWelker,schulz2015Steklov}.
	In \cite{RingWirth}, the method of steepest descent as well as a Newton’s method are introduced on infinite-dimensional manifolds.
	It also provides a BFGS quasi-Newton method as well as the convergence of a Riemannian Fletcher–Reeves conjugate gradient iteration.
	Moreover, in recent work, it has been shown that shape optimization problems can be embedded in the framework of optimization on infinite-dimensional spaces (cf., e.g., \cite{Schulz,Welker}).
	Finding a shape space and an associated metric is a challenging task and different approaches lead to various models. There exists no common shape space suitable for all applications. 
	In principal, a finite-dimensional shape optimization problem can be obtained for example by representing shapes as splines.
	However, the connection of shape calculus with infinite-dimensional spaces \cite{Delfour-Zolesio-2001,ItoKunisch,SokoZol} leads to a more flexible approach. 
	One possible approach is to define shapes as elements of a Riemannian manifold as proposed in \cite{MichorMumford2,MichorMumford1}.
	From a theoretical point of view this is attractive because algorithmic ideas from \cite{Absil} can be combined with approaches from differential geometry.

	Not every space of interest is equipped with a Riemannian manifold structure.
	For example, a quotient of Riemannian manifolds that does not inherit the Riemannian manifold structure is not a Riemannian manifold.
	Moreover, in PDE-constrained shape optimization, finite element methods (FE) are usually used to discretize the models which leads to a shape space, the space of $H^{1/2}$-shapes, which is not a manifold but a diffeological space (cf.~\cite{Kathrin_SuitableSpaces}).
	If the space of interest is not equipped with a Riemannian manifold structure, the question arises what to do if we want to optimize on this space. 
	In general, we do not have a tangent space of a non-manifold space. However, we need tangent spaces in order to define a gradient or a covariant derivative.
	Additionally, objects like tangent spaces and gradients are also needed since they form the basic ingredients for even the simplest optimization techniques like the method of steepest decent.
	This paper addresses exactly the question above what to do if we want to optimize on a space which is not a Riemannian manifold.
	In particular, we concentrate on special objects necessary for optimization techniques on spaces which are not manifolds.
	We concentrate on the so-called diffeological spaces, firstly introduced by J.M.~Souriau in the 1980s (cf.~\cite{Souriau}).
	
	Diffeological spaces can be seen as a generalization of smooth manifolds.
	There are other concepts of generalization of smooth spaces, e.g., Chen, Fr\"{o}licher or Sikorski spaces. We refer to \cite{Stacey} for these various concepts and their comparison.   
	Diffeological spaces are stable under almost every operation like taking subsets and forming quotients, co-products or products.
	Thus, diffeological spaces are one of the generalizations which include infinite-dimensional manifolds, manifolds with corners as well as a variety of spaces with complicated local behavior.
	In consequence, a diffeology defines a smooth structure on nearly any space.
	In the last decades, a lot of concepts were generalized to diffeological spaces like vector spaces, cohomology, orbifolds, tangent spaces or tangent bundle (cf.~\cite{ChristensenWu,Iglesias,Iglesias_Karshon,laubinger,torre}).
	In order to formulate the method of steepest descent on a diffeological shape space, we need various objects like gradients or retractions not defined on diffeological spaces so far.
	One  main aim of this paper is the definition of a diffeological gradient.
	The formulation of such a definition is very challenging due to the several definitions of diffeological tangent spaces which do not coincide.
	There are at least six different definitions of diffeological tangent spaces (cf.~\cite{ChristensenWu,Iglesias,Iglesias_Karshon,laubinger,torre,vincent}).
	Some of those are very general \cite{torre} or based on category theory \cite{ChristensenWu}.
	Since optimization methods deal with directional derivatives, we need to choose a tangent space definition related to some generalizations of directional derivatives.
	We generalize the usual setting of tangent spaces by considering  paths through a point.
	For our definition of tangent spaces, we give the example of real valued polynomials in one variable denoted by $\poly(\RR)$.
	Moreover, we define a diffeological Riemannian space including a diffeological gradient using our definition of a tangent space.
	Finally, we define a diffeological retraction and generalize the definition of a Levi-Civita connection in order to be able to update iterates in an optimization technique.
	We conclude the paper with the formulation and application of the steepest descent method on diffeological spaces based on the above-mentioned objects defined in the paper.
	
	This paper is structured as follows.
	In section~\ref{sec:A_brief}, we give a brief introduction into diffeological spaces. 
	Section~\ref{sec:Opt_on_ds} clarifies the difference between diffeological spaces and smooth manifolds (cf.~subsection~\ref{subsec:DS_vs_Manif})  and gives an idea of an optimization algorithm on diffeological spaces (cf.~subsection~\ref{subsec:Algo}).
	The necessary tools to generalize optimization techniques from manifolds to diffeological spaces are defined in section~\ref{section:Objects}.
	In particular, we present a diffeological gradient and related objects like a diffeological tangent space and diffeological Riemannian space (subsection~\ref{subsec:DiffeologicalGradient}). In particular, we give two examples for a diffeological tangent space.
	Moreover, a diffeological retraction as well as a diffeological Levi-Civita connection is presented in subsection~\ref{subsec:UpdateIterates}. 
	Finally, in section~\ref{sec:Star}, we formulate the gradient decent method on diffeological spaces and apply it to an example.
	We end this paper with a conclusion and outlook in section~\ref{sec:conclusion}.


\section{A brief introduction into diffeological spaces}
\label{sec:A_brief}
	In this section, we define diffeologies and related objects, which we need for the next sections.
	To make the paper self-contained, we provide most of definitions and lemmas from \cite{Iglesias} although some of them are not new. 
	For a detailed introduction into diffeological spaces we refer to \cite{Iglesias}.

	We start with the definition of a \emph{diffeology}, with which a diffeological space is equipped, and \emph{plots}, which are the elements of a diffeology. 
	Roughly speaking, a diffeological space is a non-empty set together with a set of parametrizations which have to satisfy three axioms.
	\begin{definition}[Parametrization, diffeology, diffeological space, plots]
		Let $X$ a non-empty set. A {parametrization} in $X$ is a continuous map $U\to X$, where $U$ is an open subset of $\mathbb{R}^n$.
		A {diffeology} on $X$ is any set $D_X$ of parametrizations in $X$ such that the following three axioms are satisfied:
		\begin{itemize}
			\item[(i)] {Covering:} Any constant parametrization $\mathbb{R}^n\to X$ is in $D_X$.
			\item[(ii)] {Locality:} 
			Let $P$ be a parametrization in $X$, where $\text{\emph{dom}}(P)$ denotes the domain of $P$. If for all $r\in\text{\emph{dom}}(P)$, there is an open neighborhood $V$ of $r$ such that the restriction $P_{\mid_V} \in D_X$ , then  $P \in D_X$.
			\item[(iii)] {Smooth compatibility:} Let $p\colon U_p \to X$ be an element of $D_X$, where $U_{p}$ denotes an open subset of $\mathbb{R}^n$. Moreover, let $\varphi \colon U'\to U_p$ be a smooth map in the usual sense, where $U'$ denotes an open subset of $\mathbb{R}^m$. Then $p\circ \varphi \in D_X$ holds.
		\end{itemize}
		A non-empty set $X$ together with a diffeology $D_X$ on $X$ is called a diffeological space and denoted by $(X,D_X)$. The parametrizations $p\in D_X$ are called {plots} of the diffeology $D_X$. If a plot $p\in D_X$ is defined on $U_p\subset \mathbb{R}^n$, then $n$ is called the dimension of the plot and $p$ is called {$n$-plot}.
	\end{definition}

	Note that a diffeology as a structure and a diffeological space as a set equipped with a diffeology are distinguished only formally. Every diffeology on a set contains the underlying set as the set of non-empty 0-plots. In consequence, we do not differ between a diffeological space and a diffeology.
	In the literature, there are a lot of examples of diffeologies, e.g., the diffeology of the circle, the square, the set of smooth maps, etc.

	Given a diffeology $D_X$ on a space $X$, we notice that $D_X$ is just a subset of the set of all parametrizations of $X$. 
	Thus, we are able to compare diffeologies  of the same space by a subset relation.
	If $D'$ is another diffeology on $X$ such that $D_X \subset D'$, we say that $D_X$ is {finer} than $D'$. 
	Equivalently we say that $D'$ is {coarser} than $D_X$ if $ D' \subset D_X $.

	In the following we consider maps between diffeological spaces.
	Since we want to preserve structure, we consider special maps between diffeological spaces, the so-called \emph{smooth maps}.

	\begin{definition}[Smooth map between diffeological spaces, diffeomorphism]
		Let $(X,D_X),(Y,D_Y)$ be two diffeological spaces. A map $f\colon X\to Y$ is {smooth} if for each plot $p\in D_X$, $f\circ p$ is a plot of $ D_Y $, i.e., $f\circ D_X\subset D_Y$. If $f$ is bijective and if both, $f$ and its inverse $f^{-1} $, are smooth, $f$ is called a {diffeomorphism}. In this case, $X$ is called {diffeomorphic} to $Y$.
	\end{definition}

	The stability of diffeologies under almost all set constructions is one of the most striking properties of the class of diffeological spaces, e.g., the subset, quotient, functional or powerset diffeology. 
	In the following, we concentrate on the  \emph{subset} as well as the \emph{sum diffeology}. 
	They occur naturally if we handle diffeological spaces.
	Afterwards, we define the space $\Cinf(X,Y)$ as well as a diffeology $D_{\Cinf(X,Y)}$.

	Every subset of a diffeological space carries a natural \emph{subset diffeology}, which is defined by the \emph{pullback} of the ambient diffeology by the \emph{natural inclusion}. 
	For two sets $A,B$ with $A\subset B$, the \emph{(natural) inclusion} is given by
	$\iota_A\colon A\to B$, $x\mapsto x$. The pullback is defined as follows:

	\begin{deflem}[Pullback]
		Let $X$ be a set and $(Y,D_Y)$ be a diffeological space. Moreover, $f\colon X\to Y$ denotes some map.
		\begin{itemize}
			\item[(i)] 	There exists a coarsest diffeology of $X$ such that $f$ is smooth. 
			This diffeology is called the pullback of the diffeology $D_Y$ by $f$ and is denoted by $f^\ast(D_Y)$.
			\item[(ii)] Let $p$ be a parametrization in $X$. 
			Then $p\in f^\ast(D_Y)$ if and only if $f\circ p\in D_Y$. 
		\end{itemize}
	\end{deflem}

	The construction of subset diffeologies is related to so-called \emph{inductions}.

	\begin{definition}[Induction]
		Let $(X,D_X),(Y,D_Y)$ be diffeological spaces. A map $f\colon X\to Y$ is called induction if $f$ is injective and $f^\ast(D_Y)=D_X$, where $f^\ast(D_Y)$ denotes the pullback of the diffeology $D_Y$ by $f$.
	\end{definition}

	Now, we are able to define the subset diffeology.
	\begin{deflem}[Subset diffeology]
		\label{def_subsetdiff}
		Let $(X,D_X)$ be a diffeological space and let $A\subset X$ be a subset. 
		Then, $A$ carries a unique diffeology $D_A$, called the subset or induced diffeology, such that the inclusion map $\iota_A\colon A\to X$ becomes an induction, namely, $D_A=\iota_A^\ast(D_X)$.
		We call $(A,D_A)$ the diffeological subspace of $X$.
	\end{deflem}

	\begin{remark}
		\label{Rem:Funktionen}
		For convenience of the reader, we often write $ [p \colon U_p \ra X, u \mapsto p(u)] $ instead of only $p$, where $ p \colon U_p \ra X, u \mapsto P(u)$ is a map.
		Sometimes we only write $ [p \colon U_p \ra X] $ or $ [p \colon u \mapsto p(u)] $ if the other objects are obvious.
	\end{remark}

	If a family $(X_i, D_{X_i})_{i \in J}$ of diffeological spaces is given, we are able to define a natural diffeology on
	\begin{align*}
	\coprod_{i \in J} X_i  \coloneqq \{ (i,x_i) \mid i \in J, x_i \in X_i \}.
	\end{align*}
	The diffeology $ D_{\coprod_{i \in J} X_i}$ is given by the set of parametrizations that are locally plots for some $X_i$, i.e., 
	\begin{align*}
	&[p \colon  r \mapsto (i_r, p(r)) ] \in D_{\coprod_{i \in J} X_i} \\
	&\Leftrightarrow \forall r \, \exists i \, \exists V(r)\colon  \,i_{r'} = i  \text{ for all } r' \in V(r) \text{ and } p_{\mid_V} \in D_{X_i}.
	\end{align*}
	We call the diffeology $D_{\coprod_{i \in J} X_i}$ the \emph{sum diffeology}.

	For the set of smooth functions there is a natural diffeology, the so-called \emph{functional diffeology}, on which we concentrate now.
	Let $X$, $Y$ be two diffeological spaces. 
	We denote the set of all smooth functions between $X$ and $Y$ by  $C^{\infty}(X,Y)$ . 
	The map 
	\begin{align*}
		\mathrm{ev}\colon  C^{\infty}(X,Y) \times X \rightarrow Y,\, (f,x) \mapsto f(x)
	\end{align*}
	defines the \emph{evaluation map}. 
	Every diffeology on  $C^{\infty}(X,Y)$ for which the evaluation map is smooth is called a \emph{functional diffeology}.
	Then, $p\colon U_p \rightarrow C^{\infty}(X,Y)$ is a plot if and only if the map $ U_p \times X \rightarrow Y, \,(u,x) \mapsto p(u)(x) $ is smooth, which is equivalent to 	the fact that $ U_p \times U_q \rightarrow Y,\, (u_p,u_q) \mapsto p(u_p)(q(u_q)) $ is a plot of $ Y $ for all $  q \in D_X $.
	We call this diffeology the \emph{standard functional diffeology}.
	
	Finally, we consider quotients of diffeological spaces by some relations and their diffeology.
	Let $(X,D_X)$ be a diffeological space and $\sim$ be an equivalence relation on $X$. 
	Then, the quotient set $ \sfrac{X}{\sim} $ carries a unique diffeologcial sturcture $ D_{\sfrac{X}{\sim}} $.
	The space $ (\sfrac{X}{\sim}, D_{\sfrac{X}{\sim}}) $ is called  the \emph{diffeological quotient} of $X$ by the relation $\sim$.
	A parametrization $ p \colon U_p \rightarrow  \sfrac{X}{\sim}$ is a plot if $ p $ is locally equal to a plot of $ X $, i.e., for all $ u_0 \in U_p$ exist $ V \subset U_p $ open  and $ q \in D_X $  such that  $ p(u) = q(u) $ for all $ u \in V $.
	Now, we are familiar with the basics in diffeologies such that we can handle the concepts in the next sections.


\section{Optimization techniques on diffeological spaces}
\label{sec:Opt_on_ds}
	Diffeological spaces can be seen as generalizations of manifolds. In subsection~\ref{subsec:DS_vs_Manif}, we clarify how we consider diffeological spaces as a generalization of manifolds.
	However, manifolds can be generalized in many ways. 
	In \cite{Stacey}, a summary and comparison of possibilities to generalize smooth manifolds are given. 
	We end this section with an optimization algorithm on manifolds in subsection~\ref{subsec:Algo} in order to clarify which objects for an optimization algorithms on diffeological spaces are needed.

	\subsection{Differences between diffeological spaces and manifolds}
	\label{subsec:DS_vs_Manif}
		In this subsection, we figure out the main differences between diffeological spaces and manifolds. 
		We consider manifolds as special diffeological spaces.
		For simplicity, we concentrate on finite-dimensional manifolds. However, it has to be mentioned that infinite-dimensional manifolds can also be understood as diffeological spaces. This follows, e.g., from \cite[Corollary~3.14]{KrieglMichor} or \cite{Losik}.

		Given a smooth manifold, there is a natural diffeology on this manifold consisting of all parametrizations which are smooth in the classical sense.
		This yields the following definition.

		\begin{definition}[Diffeological space associated with a manifold]
			Let $M$ be a finite-dimensional (not necessarily Hausdorff or paracompact) smooth manifold. The diffeological space associated with $M$ is defined as $(M,D_M)$, where the diffeology $D_M$ consists precisely of the parametrizations of $M$ which are smooth in the classical sense.
		\end{definition}

		\begin{remark}
			\label{remark_diff}
			If $M,N$ denote finite-dimensional manifolds, then $f\colon M\to N$ is smooth in the classical sense if and only if it is a smooth map between the associated diffeological spaces $(M,D_M)\to(N,D_N)$.
		\end{remark}

		In order to characterize the diffeological spaces which arise from manifolds, we need the concept of \emph{smooth points}. 
		For a diffeological space $(X,D_X)$, a point $x\in X$ is called smooth if there exists an open subset $U\subset X$ which is open with respect to the topology of $X$ and contains $x$ such that $(U,D_U)$ is diffeomorphic to an open subset of $\mathbb{R}^n$, where $D_U$ denotes the subset diffeology.
		The concept of smooth points is quite simple. Let us consider the coordinate axes, e.g., in $\mathbb{R}^2$. All points of the two axis with exception of the origin are smooth points.

		Now, we are able to formulate the following theorem, which connects manifolds and diffeological spaces, firstly introduced and proved in \cite{Kathrin_SuitableSpaces}.

		\begin{theorem}
			\label{theorem_diffman}
			A diffeological space $(X,D_X)$ is associated with a (not necessarily paracompact or Hausdorff) smooth manifold if and only if each of its points is smooth.
		\end{theorem}

		\noindent
		This theorem clarifies the difference between manifolds and diffeological spaces. 
		Roughly speaking, a manifold of dimension $n$ is getting by glueing together open subsets of $\mathbb{R}^n$ via diffeomorphisms. In contrast, a diffeological space is formed by glueing together open subsets of $\mathbb{R}^n$ with the difference that the glueing maps are not necessarily diffeomorphisms and that $n$ can vary.
		However, note that manifolds deal with charts and diffeological spaces deal with plots.
		A system of local coordinates, i.e., a diffeomorphism $\varphi \colon U\to U'$ with $U\subset \mathbb{R}^n$ open and $U'\subset X$ open, can be viewed as a very special kind of plot $U\to X$ which induces an induction on the corresponding diffeological spaces.

		\begin{remark}
			Note that we consider smooth manifolds which do not necessary have to be Hausdorff or paracompact. If we understand a manifold as Hausdorff and paracompact, then the diffeological space $(X,D_X)$ in Theorem~\ref{theorem_diffman} has to be Hausdorff and paracompact.
			In this case, we need the the concept of open sets in diffeological spaces.
			Whether a set is open depends on the topology under consideration.
			In the case of diffeological spaces, openness depends on the $D$-topology, which is a natural topology and introduced in \cite{Iglesias} for each diffeological space. 
			Given a diffeological space $(X,D_X)$, 
			the $D$-topology is the finest topology such that all plots are continuous. That is, a subset $U$ of $X$ is open (in the D-topology) if for any plot $p\colon U_p\to X$ the pre-image $p^{-1}(U)\subset U_p$ is open.
			For more information about the $D$-topology we refer to the literature, e.g., \cite[Chapter~2, 2.8]{Iglesias} or \cite{ChristensenWu}.
		\end{remark}

	\subsection{Towards optimization algorithms on diffeological spaces}
	\label{subsec:Algo}
	
		In the following, we would like to give an idea of an optimization algorithm on diffeological spaces.
		For this purpose, we consider the steepest descent algorithm~\ref{Algo} on Riemannian manifolds.

		\begin{algorithm}
			\caption{Steepest descent method on the complete Riemannian manifold~$(M,g)$}
			\label{Algo}
				\begin{algorithmic}
					\State \textbf{Require:} Objective function $f$ on $M$; Levi-Civita connection $\nabla$ on $(M,g)$;  \\\phantom{\textbf{Require:} }step size strategy.
					\vspace{.1cm}
					\State \textbf{Goal:} Find the solution of $\min\limits_{x\in M}f(x)$.
					\vspace{.1cm}
					\State \textbf{Input:} Initial data $x_0\in M$. 
					\vspace{.3cm}
					
					\State \textbf{for} $k=0,1,\dots$ \textbf{do}
					\vspace{.1cm}
					\State [1] Compute $\text{grad}f(x_k)$ denoting the Riemannian shape gradient of $f$ in $x_k$.
					\vspace{.1cm}
					\State [2] Compute step size $t_k$.
					\vspace{.1cm}
					\State [3] Set $			x_{k+1}:= \operatorname{exp}_{x_k}\left(-t_k \text{grad}f(x_k)\right)$, where $\mathrm{exp}\colon TM\to M$ denotes the ex- \\\phantom{[3]} ponential map
					\vspace{.1cm}
					\State \textbf{end for}
					\vspace{.3cm}
				\end{algorithmic}
		\end{algorithm}

		\begin{remark}
			\label{rem:Algo}
			In algorithm~\ref{Algo}, it is assumed that the Riemannian manifold needs to be complete with Riemannian structure $g$ and Levi-Civita connection $\nabla$.
		The completeness of a Riemannian manifold results in a global exponential map, which is then defined by the Levi-Civita connection. 
		Since performing the exponential map is an expensive operation in optimization strategies, it is often approximated and replaced by a so-called retraction mapping. 
		If we are using a retraction mapping it is possible to drop the completeness of the Riemannian manifold.
		\end{remark}

		The aim of the paper is to generalize the necessary objects in algorithm~\ref{Algo} from manifolds to diffeological spaces in order to get the steepest descent algorithm on diffeological spaces. 
		First, we notice that a \emph{complete Riemannian manifold} with a \emph{Riemannian structure} is required. 
		Thus, we also need equivalent concepts on a diffeological space.
		In addition, a \emph{gradient} as well as an \emph{exponential map} need to be specified on diffeological spaces.
		As mentioned in remark~\ref{rem:Algo}, the exponential map is generally not a local diffeomorphism in infinite-dimensional manifolds and, additionally, an expensive operation in optimization techniques. Thus, we mainly concentrate on a \emph{retraction mapping} in this paper.
		Since---assuming its existence--- an exponential map can be defined by using a Levi-Civita connection (cf.~\cite{HelgasonSigurdur} for the usual setting),  we also have a brief look on  a \emph{Levi-Civita connection} in this paper.
		Finally, we recognize that algorithm~\ref{Algo} works with \emph{smooth maps} $ f\colon M \to \RR $.
		In the steepest descent method on manifolds, we usually do not need smooth functions but due to the nature of diffeological spaces, only maps that are smooth in a diffeological sense are of interest. 
		Smooth functions between diffeological spaces are already defined in section~\ref{sec:A_brief}. The next section focuses on the remaining objects.


\section{Necessary objects for diffeological optimization techniques}
\label{section:Objects}
	
	In order to extend optimization techniques from Riemannian manifolds to diffeological spaces we need to transfer concepts known on manifolds into a diffeological setting. 
	One of the main aims of this section is the definition of a \emph{diffeological gradient}.
	The definition of a diffeological gradient is very challenging because of the several definitions of tangent spaces that do not coincide. 
	For example, some of those are very general \cite{torre}, based on category theory \cite{ChristensenWu} or based on the idea of charts \cite{laubinger}.
	
	In subsection~\ref{subsec:DiffeologicalGradient}, we set up  a scheme of a diffeological Riemannian space combined with a diffeological gradient based on specific tangent spaces suitable for optimization methods.
	Subsection~\ref{subsec:UpdateIterates} presents a \emph{diffeological retraction} and the \emph{Levi-Civita connection on diffeological spaces} in order to be able to update the iterates in an optimization algorithm.
	The investigations of the Levi-Cevita connection with view on the existence and uniqueness 
	are beyond the scope of this paper. 
	In general, the existence of a Levi-Civita connection and, thus, also of the exponential map is not guaranteed in infinite-dimensional manifold.
	Therefore, we focus on  a diffeological retraction instead of the exponential map.

	\subsection{Definition of a diffeological gradient}
	\label{subsec:DiffeologicalGradient}
	
		In order to define a diffeological gradient, we need a bunch of definitions, e.g., vector spaces, tangent spaces and tangent bundles.
		First, we concentrate on a \emph{diffeological tangent space} in subsection~\ref{subsection:DiffTangentSpace} and subsection~\ref{subsec:Examples}.
		Afterwards, we generalize the concepts of \emph{vector} and \emph{tangent bundles} from manifolds to diffeological spaces and concentrate on a \emph{diffeological Riemannian space} (subsection~\ref{subsec:DiffRiemSpace}).
		With all these concepts we are able to define a diffeological gradient at the end of subsection~\ref{subsec:DiffRiemSpace}.

	\subsubsection{Diffeological tangent space}
	\label{subsection:DiffTangentSpace}

		The definition of a diffeological tangent space is a challenging task because there does not exist a general diffeological tangent space. 
		For example, there are  two  algebraic motivated definitions in \cite{ChristensenWu},  a definition that concentrates on multilinear algebra in \cite{Iglesias}, a definition that is derived from the tangent spaces of the domains of the plots and a very general definition in \cite{torre}.
		It is an open question if some of those tangent space definitions are equivalent. However, we know that not all of them are equal to each other.
		For example, the two  tangent spaces introduced in \cite{ChristensenWu} do not coincide.
		
		Since optimization methods deal with directional derivatives, we need a tangent space definition related to some generalizations of directional derivatives. Thus, with view on optimization techniques, we need a novel definition of a tangent space.
		In the following, we define a diffeological tangent space by generalizing the usual setting of tangent spaces by considering  paths through a point and discuss its properties.

		A tangent space in the sense of a manifold is a vector space and has some important properties.
		Since the diffeological tangent space is a generalization, we first need to generalize the concepts of vector spaces to diffeological spaces.
		We introduce the so-called \emph{diffeological $\mathbb{K}$-vector spaces}, where $\mathbb{K}\in \{\mathbb{R},\mathbb{C}\}$. In \cite{Iglesias},  a diffeological space $(E,D_E) $ is called a diffeological $\mathbb{K}$-vector space if 
		$E$ is a $\mathbb{K}$-vector space and addition in $E$ as well as scalar multiplication with elements in $\mathbb{K}$ are smooth if we equip $ \mathbb{K} $ with the standard diffeology.
		The set of linear maps between two diffeological $\mathbb{K}$-vector spaces $E$ and $F$ is denoted by $L(E,F)$. 
		In contrast to manifolds, linear maps between diffeological $\mathbb{K}$-vector spaces are not consequently smooth (cf., e.g., \cite{Iglesias}).
		Thus, we denote the set of smooth linear maps between $ E $ and $ F $ by $L^{\infty}(E,F)$.
		We have
		\begin{align*}
		L^{\infty}(E,F) = L(E,F) \cap C^{\infty}(E,F).
		\end{align*}
		For every $\mathbb{K}$-vector space $E$, there exists a diffeology on $ E $ such that $ E $ becomes a diffeological  $\mathbb{K}$-vector space. This diffeology is called the \emph{fine diffeology}.
		The fine diffeology trivializes the set of linear functions. To be more precisely, if $E$ and $F$ are diffeological $\mathbb{K}$-vector spaces and $ E $ is equipped with the fine diffeology, then 
		\begin{align}
		\label{eq:sm_lin=lin}
		L^{\infty}(E,F) = L(E,F).
		\end{align}

		Next, we concentrate on derivatives of smooth maps from a diffeological space into the real numbers.
		Since a map from a diffeological space into another one is generally not differentiable in the Euclidean sense, we focus on the plots of the source space. 
		
		In order to define a tangent space $ T_xX $ at $ x \in X $ for a smooth manifold $ X $,
		one  can consider the set of all paths that maps $ 0 $ to $ x $ and identify two of these paths with each other if their derivative coincide in $ 0 $.
		Then, an element in $ T_xX $ is  a residue class $ [\gamma] $ for $ \gamma\colon  \RR \rightarrow X $ such that $ \gamma(0) = x $.
		Moreover, the directional derivative of a function  $ f\colon X \rightarrow \RR $ in direction $ v $ is given by the derivative $ \tfrac{\partial (f \circ \gamma)(t) }{\partial t}_{\mid_{t = 0}} $ for $ \gamma\colon \RR \rightarrow X $ such that $ \gamma(0) = x $ and $ \gamma '(0) = v $.
			Similar to directional derivatives on manifolds or to the definition of a tangent space on a manifold, we consider paths from $\RR$ into the diffeological space $ X $ that pass through an element $ x \in X $.
		We denote the set of all smooth paths in a diffeological space $X$ by
		\begin{align*}
			\mathrm{Paths}(X)  \coloneqq \{ \alpha \in C^{\infty}(\mathbb{R},X) \}
		\end{align*}
		and the set of all paths in $X$ centered at $x \in X$  by
		\begin{align*}
			\mathrm{Paths}_x(X)  \coloneqq \{ \alpha \in C^{\infty}(\mathbb{R},X) \mid \alpha(0) =x \}.
		\end{align*}
		Since $\mathrm{Paths}_x(X)$ is a subset of $C^{\infty}(\mathbb{R},X)$, $\mathrm{Paths}_x(X)$ equipped with the functional diffeology is a diffeological subspace of $C^{\infty}(\mathbb{R},X)$.
		Now, we can define the so-called \emph{path derivative} inspired by the usual directional derivative.
		
		\begin{definition}[Path derivative]
			\label{def:PathDerivative}
			Let $X$ be a diffeological space and $\alpha \in \mathrm{Paths}_x(X)$. The {path derivative}  in direction $\alpha$ is defined by 
			\begin{align}
			\label{PathDerivative}
			\d_{\alpha}\colon  C^{\infty}(X,\mathbb{R}) \rightarrow \mathbb{R},\, f \mapsto
			 \d_{\alpha}(f)   \coloneqq \dfrac{\partial}{\partial u}\left(f(\alpha(u)\right)_{\mid_{u=0}}\in \mathbb{R}.
			\end{align}
			One calls $\d_{\alpha}(f) $ the \emph{path derivative} of the function $f\in C^{\infty}(X,\mathbb{R})$  in direction $\alpha$.
		\end{definition}
		
		Inspired by the definition of a tangent space of a manifold, it would be natural to define a diffeological tangent space at a point $x$ as the set of all path derivatives through $x$, where two paths are identified by each other if their derivative coincide in  $ 0 $.
		Unfortunately, we are generally not able to derive an element in $ 	\mathrm{Paths}_x(X) $, i.e., a (diffeological) smooth function from $\mathbb{R}$ into a diffeological space $X$.
		Thus, we need a concept of a tangent space without such an identification.
		For this, we  introduce a first \emph{tangent cone} of a diffeological space $X$ at $x\in X$ by
			\begin{align*}
			\tilde{C}_xX  \coloneqq \{ \mathrm{d}_{\alpha} \mid \alpha \in \mathrm{Paths}_x(X) \},
			\end{align*}
			which is firstly defined in \cite{vincent}.
			In the following, we say $p\colon  U_p \rightarrow \tilde{C}_xX$ with $U_{p}$ denoting an open subset of $\mathbb{R}^n$ is a plot if it is locally equal to  $\mathrm{d}_{\gamma}$ for a plot $\gamma$ of $\mathrm{Paths}_x(X)$.  
			The diffeology that is given by these plots is called the \emph{tangent cone diffeology}.

			In order to obtain a \emph{diffeological tangent space} from the tangent cone $ \tilde{C}_xX $, we consider the span of the tangent cone as follows:
			\begin{itemize}
				\item 
				An element in $\mathrm{span}(\tilde{C}_xX)$ is a finite sum of elements in $ \tilde{C}_xX $ that are multiplied by a scalar.
				\item
				A plot of $\mathrm{span}(\tilde{C}_xX)$ isgiven by a mapping
				\begin{align*}
					U\to \mathrm{span}(\tilde{C}_xX), \, u \mapsto \sum_{i=1}^{n}\lambda_i(u)p_i(u)
				\end{align*}
				with $U\subset\mathbb{R}^n$ denoting an open subset with $n \in \mathbb{N}$, $\lambda_i\colon  U \rightarrow\mathbb{R}$ is smooth and $p_i \colon U \rightarrow \tilde{C}_xX$ is a plot of $\tilde{C}_xX$.
			\end{itemize}
		The diffeology introduced on $ \mathrm{span}{\tilde{C}_xX} $ is defined  in \cite{vincent} and called \emph{the weak diffeology}.
		Equipped with the weak diffeology, $ \mathrm{span}(\tilde{C}_xX) $ becomes the finest possible diffeological vector space, such that $ D_{\tilde{C}_xX} \subset D_{\mathrm{span}(\tilde{C}_xX)}$. 
		Therefore, in \cite{vincent}, the \emph{(diffeological) tangent space} of $X$ at $x$ is defined by
			\begin{align}
			\label{def:TxX_tilde}
			\tilde{T}_xX  \coloneqq \mathrm{span}\{\tilde{C}_xX\}.
			\end{align}
		Similar to the usual setting on smooth manifolds, for $ x \in X $ the tangent map of a smooth map $ f \colon X \rightarrow Y $ between two diffeological spaces $ X,Y $ is given by
			\begin{align*}
			\tilde{T}_xf: \tilde{T}_xX \rightarrow \tilde{T}_{f(x)}Y,\,
			\d_{\alpha}		\mapsto 		\d_{f \circ \alpha},
			\end{align*}
		which is a smooth map (cf. \cite{vincent}).
		Thanks to the definition of the tangent map we are able to find an identification of elements in $ \mathrm{Paths}_x(X) $ similar to the one on manifolds.
		\begin{definition}
			\label{def:Identification}
			Let $ X $ be a diffeological space, $ x \in X $ and $ \alpha, \beta \in \mathrm{Paths}_x(X) $. We define the relation $ \sim $ by
			\begin{align*}
			\alpha \sim \beta \quad  \Leftrightarrow \quad \tilde{T}_0 \alpha = \tilde{T}_0 \beta.
			\end{align*}
		\end{definition}
		With the identification in definition~\ref{def:Identification}, we are now able to define a tangent space that is more similar to the definition of a tangent space on manifolds than the tangent space given in (\ref{def:TxX_tilde}). 
		In order to define it, we need a novel version of the tangent cone.
		
		\begin{definition}[Tangent cone]
			\label{def:CxX}
			Let $ X $ be a diffeological space and $ x $ an element in $ X $. Then the tangent cone is defined by
			\begin{align*}
				C_xX \coloneqq \{ \d_{\alpha} \mid \alpha \in \sfrac{\mathrm{Paths}_x(X)}{\sim}  \}.
			\end{align*}
		\end{definition}
		
		A parametrization $ p \colon U_p \rightarrow C_xX $ is a plot if and only if $ p(u) = \d_{q(u)}  $ holds locally for a plot $ q \colon U_q \rightarrow  \sfrac{\mathrm{Paths}_x(X)}{\sim}$. Now, we can define the diffeological tangent space analogue to $ \tilde{T}_xX $.
		\begin{definition}[Diffeological tangent space]
			\label{def:TxX}
			The tangent space of a diffeological space $X$ in $x\in X$ is defined by
			\begin{align}
			\label{TangentSpace}
				T_xX \coloneqq \mathrm{span}(C_xX).
			\end{align}
			The cotangent space of $T_xX$ is the space of smooth functionals of $T_xX$, in symbols
			\begin{align*}
			T_x^*X  \coloneqq L^{\infty}(T_xX,\mathbb{R}).
			\end{align*}
		\end{definition}
		
		We have the following diffeomorphism:
		
		\begin{theorem}
			\label{The:Isomorphism}
			Let $ X $ be a diffeological space and $ x \in X$. The diffeological tangent spaces $T_x X$ and $\tilde{T}_xX$ are diffeomorphic. 
		\end{theorem}
		
		\begin{proof}
			Thanks to the quotient diffeology, a parametrization $ p\colon U_p \rightarrow C_xX $ is a plot if and only if $ p(u) =  \d_{r(u)}$ holds locally for a plot $ r $ of $ \mathrm{Paths}_x(X) $.
			Thus, $ p $ is a plot of $ C_xX $ if and only if $ p $ is a plot of $ \tilde{C}_xX $.
			This gives that $ \tilde{C}_xX $ and $ C_xX $ coincide.
		\end{proof}
		
		Thanks to theorem~\ref{The:Isomorphism}, we know that the diffeological tangent spaces defined in (\ref{def:TxX_tilde}) and (\ref{TangentSpace}) are equal to each other. However, in contrast to $T_xX $, where an element $ \d_{\alpha} \in T_xX $ is given by exactly one element in $ \sfrac{\mathrm{Paths}_x(X)}{\sim} $, an element in $\tilde{T}_xX$ can generally be obtained by more than one element in $ \mathrm{Paths}_x(X) $.

		\subsubsection{Examples of diffeological tangent spaces}
		\label{subsec:Examples}
		
			In the previous subsection, the diffeolgocial tangent space is defined. This subsection is devoted to two examples of diffeological tangent spaces.
			The first one is the axis in $\RR^2$ equipped with the subset diffeology.
			The set of polynomials in one real variable is the second example. We prove that there is a natural diffeology on this set of polynomials and that the tangent space at a point is given by $\RR^{\infty}$.

		\paragraph{Example 1: The cross.}
		\label{example1}
			We consider the set 
			\begin{equation}
			\label{cross}
			\mathcal{X}  \coloneqq \{ (x,y) \in \RR^2 \mid xy = 0\}
			\end{equation}
			 in $\RR^2$.
			If we equip $\mathcal{X}$ with the subset diffeology of $\RR^2$, we obtain
				\begin{align*}
					C_x\mathcal{X} = \{ r \dal \mid r \in \RR \text{, } \alpha\colon  \RR \ra \RR^2 \text{ is smooth in the Euclidean sense with } \alpha(0) = x \}.
				\end{align*}	
			From this follows that the tangent space at $x \in \mathcal{X}$ is just $\RR$ with the standard diffeology if $x \neq 0$ and the usual $\RR^2$ with the standard diffeology for $x = 0$.

		\paragraph{Example 2: Polynomials.}
		\label{example2}
			We consider the set $M$ of all polynomials from $\RR$ to $\RR$.
			In order to define a diffeology on $M$, we  first need to equip this set with a topological structure.
			\begin{definition}[Final topology]
				Let a set $X$, a family of topological spaces $X_i$ and functions $f_i\colon  X_i \rightarrow X$ be given. The final topology on $X$ is the finest topology such that the function $f_i$ are continuous.
			\end{definition}

			For simplicity, we denote
			the  set of polynomials of degree  lower or equal $n$  by
			\begin{align*}
			M_n :=\mathrm{Poly}_n(\mathbb{R})  \coloneqq \{ P \in M \mid \mathrm{deg}(P) \leq n \} \subset M,
			\end{align*}
			where $\text{deg}(P)$ denotes the degree of a polynomial $P$.
			
			We equip M with the final topology through the natural inclusions $f_i\colon M_i \to M.$
			We also know that $M_n \subset M \subset C^{\infty}(\RR)$ and, thus, $M_n$ and $M$ are naturally equipped with the subset diffeology inhered by the functional diffeology on $\Cinf(\RR)$.
			In \cite{Iglesias}, we find the following result, which we use to obtain the tangent space of $ M $.

			\begin{lemma}
				\label{le:M_n=RN}
				If $M_n = \poly_n (\RR)$ is equipped with the functional diffeology, there exists a diffeomorphism between $M_n$ and $\RR^{n+1}$. In other words, 
				\begin{align*}
				M_n \simeq \RR^{n+1}
				\end{align*}
				in a diffeological sense.
			\end{lemma}

			Since $\dal (f) = \tfrac{\d}{\d t}(f(\alpha(t)))_{\mid_{t=0}}$ for $f \in \Cinf(M)$, we investigate $\alpha$ in a neighborhood at $0 \in \RR$.
				Let  $ K \subset \RR$ be compact with $0 \in K$. Then, $\alpha_{\mid_K}\colon  K \ra M$ is a compact subset of $M$ because $\alpha_{\mid_K}$ is continuous. 
				In the following, we consider compact subsets of $M$.

			\begin{lemma}
				\label{le:cpt_subsets_in_M}
				Let $C$ be a compact subset of $M$. Then, there exists an $n \in \NN$ such that $C \subset M_n$.
			\end{lemma}
			
			\begin{proof}
				Let $x \in M$. Then, $x$ is given by
				\begin{align*}
					x(y) = \sum_{n \in \NN_0}x_ny^n,
				\end{align*}
				where $ (x_n)_{n \in \NN}$ denotes the sequence of all coefficients of $x$.
				We define the sets
				\begin{align*}
				 	\mathcal{U}_r  \coloneqq \{ x \in M \mid   |x_n| < r\quad \forall n \in M\}
				\end{align*}
				which form an open covering of $M$ with $\mathcal{U}_n \subset \mathcal{U}_m$ for $m > n$.
				Since $ C $ is compact, there exists an $R \in \NN$ such that 
				\begin{align*}
					C \subset \bigcup_{r \geq 0}^R \mathcal{U}_r = \mathcal{U}_R.
				\end{align*}	 
				We also define the map
				\begin{align*}
					\chi\colon  \NN \ra \RR^+,\, n \mapsto \sup_{x \in C} \abs{x_n}.
				\end{align*}
				Then, we get 
				\begin{align*}
					0 \leq \chi(n) \leq R \qquad \forall \, n \in \NN
				\end{align*}
				because $C \subset \mathcal{U}_R$.
				
				Now, we assume that there exists a sequence $(\eta_n)_{n \in \NN}$ such that $\eta_i < \eta_{i+1}$ and $0 < \chi(\eta_i) $ for all $i \in \NN$.
				We define
				\begin{align*}
					 \mathcal{V}_r  \coloneqq \{ x \in M \mid | x_{\eta_j} | < \tfrac{1}{2}\chi(\eta_j)\quad  \forall j > r  \}
				\end{align*}
				which also form an open covering of $M$ and satisfy $\mathcal{V}_n \subset \mathcal{V}_m$ for $m> n$.
				In a consequence, we know that there exists an $R$ such that $K \subset \mathcal{V}_R$.  
				Thus,
				\begin{align*}
					| x_{\eta_j} | < \tfrac{1}{2} \chi(\eta_j) = \tfrac{1}{2} \sup_{x \in C} |x_{\eta_j}| \qquad \forall \, j > R \text{ and } \forall \, x \in C.
				\end{align*}
				This contradicts to the definition of the supremum.
				We deduce that there exists an $n \in \NN$ such that for all $x \in C$ we have $\mathrm{deg}(x) \leq n$. 
			\end{proof}

				With the notation
				\begin{align*}
					F_n  	&\coloneqq \{ \alpha_{\mid_K}\colon K \ra \RR^n \mid \alpha \in \sfrac{ \pathsx(M)}{\sim} \}
				\end{align*}
				for a compact subset $K$ of $\RR$ with $0 \in K^{\circ}$
			we get the following result for the tangent cone of the set $M$ of all polynomials from $\RR$ to $\RR$.
			\begin{lemma}
				\label{le:$C_xM = U(...|Tn)$}
				The identity $C_xM =  \bigcup_{n > \mathrm{deg}(x)} \{ \dal \mid \alpha \in F_n \}$ holds.
			\end{lemma}
			
			\begin{proof}
					For $\d_{\alpha}$ with $\alpha \in \sfrac{\pathsx(M)}{\sim}$ it is only relevant how $\alpha$ behaves at $0$.
					Therefore, we consider a compact set $K \subset \RR$ such that $0 \in K^{\circ}$ and restrict $\alpha$ to $K$.
					This leads to
					\begin{align*}
						C_xM 	= \{\d_{\alpha_{\mid_K}} \mid \ \alpha \in \sfrac{\pathsx(M)}{\sim}\}.
					\end{align*}
					Since $\alpha$ is continuous, we obtain that $\alpha(K)$ is compact.
					We can identify $\alpha_{\mid_K}$  by a map from $K \subset \RR$ into $\RR^n$ for some $n \in \NN$ due to lemma~\ref{le:cpt_subsets_in_M}.
					Thus, we get
					\begin{align*}
						C_xM = \{\d_{\alpha} \mid \alpha \in F\},
					\end{align*}
					where
					$F 		\coloneqq \bigcup_{n > \mathrm{deg}(x)} F_n$.
			\end{proof}
			
				In the proof of lemma~\ref{le:$C_xM = U(...|Tn)$}, we notice that $\alpha_{\mid_K} \colon K \to M$ can bee seen as a map $\alpha_{\mid_K} \colon K \to \RR^n$ for some natural number $n$. 
				In consequence, we are able to consider the Euclidean derivative of $\alpha_{\mid_K}$ near zero.
				This consideration leads to the next  lemma.
			
			\begin{lemma}
				\label{le:(...|Tn) = Rn}
					The set $\{ \dal \mid \alpha \in F_n \}$ consists of directional derivatives, i.e.,
					\begin{align*}
						\{ \dal \mid \alpha \in F_n \} \simeq \RR^{n}.
					\end{align*}
			\end{lemma}
			
			\begin{proof}
				
				We define 
				\begin{align*}
				\phi\colon   \RR^n \ra \{ \dal \mid \alpha \in F_n \},\, v \mapsto \d_{[t \mapsto tv + x]}
				\end{align*}
				and show that $\phi$ is a diffeomorphism, i.e., 
				$\phi$ is bijective and both, $\phi$ and $\phi^{-1}$, are smooth.

						\emph{Bijectivitiy of $\phi$.} 	
							The map $\phi$ is surjective 
								due   lemma \ref{le:$C_xM = U(...|Tn)$} and the fact that $ \alpha \sim [t \mapsto t \alpha '(0) + x] $ holds. 
						Moreover, $\phi$ is injective because $\phi^{-1}\colon  \dal \mapsto \alpha'(0)$ is  surjective.
						
						\emph{Smoothness of $\phi$.}
							 	The map $\phi$ is obviously linear. 
							 	Thus, $ \phi $ is smooth by equation (\ref{eq:sm_lin=lin}) since $ \RR^n $ is fine (cf.~\cite{Iglesias}).
						 
						\emph{Smoothness of $\phi^{-1}$.} 
							Let $P\colon  U_P \ra \{ \dal \mid \alpha \in F_n \}$ be a plot of $\{ \dal \mid \alpha \in F_n \}$. By the subset diffeology, $P$ is a plot of $C_xM$ with values in $\{ \dal \mid \alpha \in F_n \}$. Thus, $P$ is locally equal to $\d_Q$ for a plot $Q$ of $F_n$.
								We identify $\dal \in C_xM$ with $\d_{[tv+x]}$ for $v = \alpha'(0)$. Then, for every $Q(u)$, there exists a vector $v_u = (Q(u))'(0)$ such that 
								\begin{align*}
									\d_{Q(u)} = \d_{[t \mapsto tv_u + x]}.
								\end{align*}
								Since $U_q \rightarrow M,\ u \mapsto Q(u)$ is smooth, the map $\RR \rightarrow M, t \mapsto tv_u +x$  is also smooth.
								Due to the equality
								\begin{align*}
									[u \mapsto v_u] = [u \mapsto \phi^{-1}(\d_{Q(u)})]
								\end{align*}
								the map $u \mapsto \phi^{-1}(\d_{Q(u)})$ is smooth.
								We obtain the smoothness of $\phi^{-1}$ because $u \mapsto Q(u)$ is a plot $\{ \dal \mid \alpha \in F_n \}$.
			\end{proof}

			We obtain the following result for the tangent cone  by applying lemma~\ref{le:(...|Tn) = Rn}.
			\begin{lemma}
				$C_xM \simeq \mathbb{R}^{\infty}$ holds for the tangent cone of the set $M$ of all polynomials from $\RR$ to $\RR$.
			\end{lemma}

			Since the tangent cone is obviously a diffeological vector space, we obtain with definition~\ref{def:CxX} and definition~\ref{def:TxX} the tangent space of the set of all polynomials from $\RR$ to $\RR$.
			\begin{theorem}
				\label{TxM}
				The tangent space at $x$ in $M$ is given by
				\begin{align*}
				T_xM \simeq C_xM \simeq \RR^{\infty}.
				\end{align*}
			\end{theorem}

		\subsubsection{Diffeological Riemannian space}
		\label{subsec:DiffRiemSpace}
	
			The diffeological Riemannian space is very important for generalizing optimization techniques on manifolds.
			For example, a diffeological gradient is necessary for formulating the steepest descent method on a diffeological space.
			In order to define a diffeological Riemannian space and a diffeological gradient, we first need to generalize the concepts of \emph{vector} and \emph{tangent bundles} from manifolds to diffeological spaces.
			
			First, we clarify the meaning of a \emph{bundle}, \emph{fiber} and \emph{pre-bundle} in a diffeological setting  (cf. \cite{pervova}).
			For a diffeological space $B$, a \emph{bundle} over $B$ is a diffeological space $E$ together with a subduction $\pi\colon  E \ra B$.
			We consider the bundle $\pi \colon E \to B$. The \emph{fiber} over $b \in B$ is the set $E_b\coloneqq\pi^{-1}(b)$.
			Let $(E_b)_{b \in B}$ be a collection of sets. Let us consider $E  \coloneqq \cup_{b \in B}E_b$ and the map $\pi\colon  E \ra B$ such that $\pi(E_b) = b$.
			We call $(E,B,\pi)$ a \emph{pre-bundle} with fibers $E_b$.
			A \emph{bundle diffeology} $D_E$ for $E$ is any diffeology, such that $\pi^*(D_E) = D_B$.
			
			Now, we are able to define a diffeological vector bundle and a weak vector bundle diffeology.
				A diffeological bundle, where each fiber has the structure of a diffeological vector space, is called \emph{diffeological vector bundle}.
			If $(E,B,\pi)$ denotes a pre-bundle, where $B$ denotes a diffeological space and the bundle diffeology is denoted by $D_E$, the \emph{weak vector bundle diffeology} on $E$ is generated by plots of the form
				\begin{align*}
					u \mapsto \sum_{i=1}^n \lambda_i(u) p_i(u)
				\end{align*}
				for $n \in \NN$, $\lambda\colon  U \ra \RR$ smooth and $p_i \in D_E$ such that $\pi \circ p_i = \ldots = \pi \circ p_n$.
				As mentioned, e.g., in \cite{vincent}, the weak vector bundle diffeology is the finest vector bundle diffeology generated by $D_E$.
			
			In order to define a diffeological  \emph{tangent bundle} 
				of a diffeological space $X$ we consider the \emph{tangent cone bundle} of $X$.
			It is defined as the union of all tangent cones, i.e.,
			\begin{align}
			\label{tangentcone}
				CX  \coloneqq \bigcup_{x \in X}C_xX.
			\end{align}
			Let $D_{CX}$ be the set of parametrizations $p\colon U_p \ra CX$ of $ CX $ which are locally of the form $u \mapsto \d_{\gamma(u)}$ for a smooth  $\gamma\colon  U \ra \paths(X)$. 
			Then, the set $D_{CX}$ is a \emph{bundle diffeology} for $CX$.
			Since we have the diffeological tangent cone bundle $(CX,D_{CX})$, we are able to define the diffeological tangent bundle $(TX,D_{TX})$.
				The \emph{diffeological tangent bundle} of a diffeological space $X$
				is given by
				\begin{align*}
					TX  \coloneqq \bigcup_{x \in X}T_xX.
				\end{align*}
				Together with the weak vector bundle diffeology generated by $D_{CX}$ denoted by $D_{TX}$, $(TX,D_{TX})$ is a diffeological vector bundle (cf., e.g., \cite{vincent}).
			The generating plots of $D_{TX}$ are given by
				\begin{align*}
					U \to TX,\  u \mapsto \sum_i^n\lambda_i(u)\d_{\gamma_i(u)},
				\end{align*}
				where $n \in \NN$, $\lambda_i\colon  U \ra \RR$ smooth and $\gamma_i\colon  U \ra \pathsx(X)$ with $\gamma_1(u)(0)= \ldots = \gamma_n(u)(0)$ .

			Now, we have all tools that we need to define a diffeological version of a Riemannian space with which we are able to specify the diffeological gradient.
			We start with the definition of a \emph{diffeological Riemannian space} and a \emph{diffeological Riemannian metric}.
			\begin{definition}[Diffeological Riemannian space]
				Let $X$ be a diffeological space. We say $X$ is a {diffeological Riemannian space} if there exists a smooth map 
				\begin{align*}
					g\colon  X \ra \mathrm{Sym}(TX,\RR),\, x \mapsto g_x
				\end{align*}
				such that
				\begin{align*}
					g_x\colon  T_xX \times T_xX \ra \RR
				\end{align*}
				is smooth, symmetric and positive definite.
				Then, we call the map $g$ a diffeological Riemannian metric.
			\end{definition}

			We define a gradient similar to gradients on Riemannian manifolds.
			\begin{definition}[Diffeological gradient]
				Let $X$ be a diffeological Riemannian space. The diffeological gradient $ \operatorname{grad}f $ of a function $f \in \Cinf(X)$ in $x \in X$ 	is defined as the solution of 
				\begin{align*}
					g_x(\operatorname{grad} f, \dal) = \dal(f).
				\end{align*}
			\end{definition}

			We conclude this section with an example.
			We concentrate on our first example, the cross $\mathcal{X}$ defined in (\ref{cross}).
			As seen above, the tangent space of $\mathcal{X}$ at $x\in\mathcal{X}$ is given by
			$\RR^2$
			with the standard diffeology on $\RR^2$.
			That means $P\colon U_P \ra \RR^2$ is a plot if and only if $P$ is smooth in the usual sense.
			This in mind we notice that $T_x\mathcal{X}$ is a diffeological Riemannian space.
			Indeed, the mapping $g_x$ is equal to the usual scalar product on $\RR^2$ for every $x \in \mathcal{X}$.
			 Consequently, $ g_x $ is smooth, symmetric, positive definite.
			 Furthermore it satisfies
			\begin{align*}
				T_x\mathcal{X} \overset{\sim}{\longrightarrow} T_x^*\mathcal{X},\, v \mapsto \langle v,\cdot \rangle.
			\end{align*}
			Furthermore, we notice that 
			\begin{align*}
				g\colon  \mathcal{X} \ra \mathrm{Sym}(T\mathcal{X},\RR),\, x \mapsto g_x
			\end{align*}
			is smooth.
			Thus, we are able to define a gradient on $T_x\mathcal{X}$, which is the usual gradient in the Euclidean space $\RR^2$.

	\subsection{Towards updates of iterates: Diffeological Levi-Civita connection and diffeological retraction}
	\label{subsec:UpdateIterates}
	
		In algorithm~\ref{Algo}, we need the concept of the exponential map and its approximation, the so-called retraction, to locally reduce an optimization problem on a manifold to an optimization problem on its tangent space.
		If we consider Riemannian manifolds, it is not guaranteed that the exponential map is a local diffeomorphism; there might even not to be any exponential map at all  \cite{MilnorDeWittStora}.
		However, for a Riemannian manifold there exists an exponential map if there is a Levi-Civita connection on that manifold (cf.~\cite{HelgasonSigurdur}).
		Thus, we define a diffeological version of a Levi-Civita connection, which should lead---assuming its existence---to a diffeological exponential map.

		\begin{definition}[Levi-Civita connection]
			\label{def:lc_connection}
			Let $X$ be a diffeological space with diffeological Riemannian metric $ g  $, tangent space $ T_xX $ at $ x \in X $ and tangent bundle $ TX $. We define 
				\begin{align*}
			\Gamma(TM) \coloneqq \{ \mathcal{X} \colon X \to TX \mid \mathcal{X}(p)  \in T_pX \text{ for } p \in X \}.
			\end{align*}
			A diffeological Levi-Civita connection of $ X $ is a map
			\begin{align*}
				\nabla: \Gamma(TX) \times \Gamma(TX) \to \Gamma(TX),\ 
			 (\mathcal{X},\mathcal{Y}) \mapsto \nabla_{\mathcal{X}}\mathcal{Y}
			\end{align*}
			with the following properties:
			\begin{itemize}
				\item[(i)]
					For $ f,g \in \Cinf(X),\ \mathcal{X,Y,Z} \in \Gamma(TX)$ the following equation holds:
					\begin{align*}
						\nabla_{fX + gY}Z = f\nabla_{X}Z + g\nabla_{Y}Z
					\end{align*}
				\item[(ii)]
					For $ a,b \in \RR,\ \mathcal{X,Y,} \in \Gamma(TX)$ the following equation holds:
					\begin{align*}
						\nabla_{X}(aY+bZ) = a\nabla_{X}Y + b \nabla_{X}Z
					\end{align*}
				\item[(iii)]
					For $ f \in \Cinf(X),\ \mathcal{X,Y} \in \Gamma(TX)$ the following equation holds for all $p \in X$:
					\begin{align*}
						\left( \nabla_{f\mathcal{X}}\mathcal{Y}\right) (p) 
						= f(p)\left( \nabla_{\mathcal{X}}\mathcal{Y}\right) (p) \
					\end{align*}
				\item[(iv)]
					For $ f \in \Cinf(X),\ \mathcal{X,Y} \in \Gamma(TX)$ the following equation holds for all $p \in X$:
					\begin{align*}
						\left( \nabla_{\mathcal{X}} \left( f\mathcal{Y}\right)  \right)(p)
						= \mathcal{X}(f)\mathcal{Y}(p) + f(p)\left(\nabla_{\mathcal{X}}\mathcal{Y} \right)(p) 
					\end{align*}
				\item[(v)]
					For $ \mathcal{X,Y,Z} \in \Gamma(TX) $ the following equation holds for all $p \in X$:
					\begin{align*}
						\mathcal{Z}(p)(g(\mathcal{X},\mathcal{Y})) 
						= g_p((\nabla_{\mathcal{Z}}\mathcal{Y})(p), \mathcal{Y}(p)) + g_p(\mathcal{X}(p), (\nabla_{\mathcal{Z}}\mathcal{Y})(p))
					\end{align*}
				\item[(vi)]
					For $ f \in \Cinf(X)$ the following equation holds for all $p \in X$:
					\begin{align*}
						(\nabla_{\mathcal{X}}\mathcal{Y}-\nabla_{\mathcal{Y}}\mathcal{X})(f)(p) 
						= \mathcal{X}(p)(\mathcal{Y}(f)) - \mathcal{Y}(p)(\mathcal{X}(f))
					\end{align*}
			\end{itemize}
		Here, $ \mathcal{X}(f) $ and $ \mathcal{Y}(f) $ are given by
			\begin{align*}
				\mathcal{X}(f) \colon M \to \RR, q \mapsto \mathcal{X}(q)(f) \quad \text{ and } 	\mathcal{Y}(f) \colon M \to \RR, q \mapsto \mathcal{Y}(q)(f),
			\end{align*}
			respectively.
		\end{definition}

		In general, the existence of a diffeological Levi-Civita connection is not guaranteed.
		Even on smooth manifolds, a Levi-Civita connection does not need to exist.
		However, we have the following result of uniqueness.

		\begin{lemma}
			Let $X$ be a diffeological Riemannian space with a Riemannian metric $g$ and a Levi-Civita connection $\nabla$.
			Then, $\nabla$ is unique.
		\end{lemma}
		
		\begin{proof}
			Thanks to property (v) of definition~\ref{def:lc_connection}, we have
			\begin{align*}
				\tag{1}
				\mathcal{Z}(g(\mathcal{X},\mathcal{Y})) = g(\nabla_{\mathcal{Z}}\mathcal{X},\mathcal{Y}) + g(\mathcal{X},\nabla_{\mathcal{Z}}\mathcal{Y})
			\end{align*}
			for all $\mathcal{X},\mathcal{Y},\mathcal{Z} \in \Gamma(TM)$.
			Therefore, we also have
			\begin{align*}
				\tag{2}
				\mathcal{Y}(g(\mathcal{Z},\mathcal{X})) 	& = g(\nabla_{\mathcal{Y}}\mathcal{Z},\mathcal{X}) + g(\mathcal{Z},\nabla_{\mathcal{Y}}\mathcal{X}), \\
				\tag{3}
				\mathcal{X}(g(\mathcal{Y},\mathcal{Z})) 	& = g(\nabla_{\mathcal{X}}\mathcal{Y},\mathcal{Z}) + g(\mathcal{Y},\nabla_{\mathcal{X}}\mathcal{Z}).
			\end{align*}
			
			Subtracting (2) from the sum of (1) and (3) yields
			\begin{align*}
				&\mathcal{Z}(g(\mathcal{X},\mathcal{Y})) + \mathcal{X}(g(\mathcal{Y},\mathcal{Z})) - \mathcal{Y}(g(\mathcal{Z},\mathcal{X}))\\
				& = g(\nabla_{\mathcal{Z}}\mathcal{X},\mathcal{Y}) + g(\mathcal{X},\nabla_{\mathcal{Z}}\mathcal{Y}) + g(\nabla_{\mathcal{X}}\mathcal{Y},\mathcal{Z}) + g(\mathcal{Y},\nabla_{\mathcal{X}}\mathcal{Z})\\
				&\phantom{= }\,\, - g(\nabla_{\mathcal{Y}}\mathcal{Z},\mathcal{X}) + g(\mathcal{Z},\nabla_{\mathcal{Y}}\mathcal{X})\\
				& = g\left(\nabla_{\mathcal{Z}} \mathcal{X}+\nabla_{\mathcal{X}} \mathcal{Z}, \mathcal{Y}\right)+g\left(\mathcal{X}, \nabla_{\mathcal{Z}}\mathcal{Y}-\nabla_{\mathcal{Y}} \mathcal{Z}\right)+g\left(\mathcal{Z}, \nabla_{\mathcal{X}} \mathcal{Y}-\nabla_{\mathcal{Y}} \mathcal{X}\right).
			\end{align*}
			Therefore we obtain
			\begin{align*}
				g\left(\nabla_{\mathcal{Z}} \mathcal{X}+\nabla_{\mathcal{X}} \mathcal{Z}, \mathcal{Y}\right) =	&-g\left(\mathcal{X}, \nabla_{\mathcal{Z}}\mathcal{Y}-\nabla_{\mathcal{Y}} \mathcal{Z}\right) 
				- g\left(\mathcal{Z}, \nabla_{\mathcal{X}} \mathcal{Y}-\nabla_{\mathcal{Y}} \mathcal{X}\right)\\
				&+ \mathcal{Z}(g(\mathcal{X},\mathcal{Y})) + \mathcal{X}(g(\mathcal{Y},\mathcal{Z})) - \mathcal{Y}(g(\mathcal{Z},\mathcal{X})).
			\end{align*}
			Since 
			\begin{align*}
				g\left(\nabla_{\mathcal{Z}} \mathcal{X}+\nabla_{\mathcal{X}} \mathcal{Z}, \mathcal{Y}\right) = 2g(\nabla_{\mathcal{Z}}\mathcal{X},\mathcal{Y}) - g(\mathcal{Z}\mathcal{X}-\mathcal{X}\mathcal{Z},\mathcal{Y}),
			\end{align*}
			holds, the combination with property (vi) of definition~\ref{def:lc_connection} yields 
			\begin{align*}
				2g(\nabla_{\mathcal{Z}}\mathcal{X},\mathcal{Y}) 
				=	& \mathcal{Z}(g(\mathcal{X},\mathcal{Y})) + \mathcal{X}(g(\mathcal{Y},\mathcal{Z})) - Y(g(\mathcal{Z},\mathcal{X}))\\
					&+ g(\mathcal{Z}\mathcal{X}-\mathcal{X}\mathcal{Z},\mathcal{Y}) 
					 - g(\mathcal{X}, \mathcal{Z}\mathcal{Y}-\mathcal{Y}\mathcal{Z})
					 - g(\mathcal{Z}, \mathcal{X}\mathcal{Y}-\mathcal{Y}\mathcal{X}).
			\end{align*}

			Here, the right-hand side does not depend on $\nabla$, which completes the proof.
		\end{proof}

		Since the exponential map is an expensive operation in optimization strategies, one often approximates it by a so-called \emph{retraction}. 
		Thus, we concentrate on the concept of a retraction in the following.
		For a diffeological space $ X $  with tangent space $ T_xX $ at $ x\in X $ we consider the tangent space
		\begin{align*}
			T_0(T_xX) = \{ \d_{\gamma} \mid \gamma \in \sfrac{\pathsx(T_xX)}{\sim} \}
		\end{align*}
		at $ 0 \in T_xX $.
		We start with a definition such that we assign a path in $ X $ to an element in $ x\in X $.
		For this, we define a smooth map 
		$	\overline{x} \colon \RR \to X,\, t \mapsto x$
		for a diffeological space $ X $ and an element $ x \in X $.
		By definition, we have $ \overline{x} \in \mathrm{Path}_x(X) $.
		Thus, we are able to consider $ \d_{\overline{x}} \in T_xX$.
		If $ f \in \Cinf(X) $, then 
		\begin{align*}
			\d_{\overline{x}}(f) 
				= \dfrac{\partial}{\partial u}f(\overline{x}(u)) 
				= \dfrac{\partial}{\partial u}f(x)
				= 0.
		\end{align*}
		In consequence, $ \d_{\overline{x}} = \d_0 $ holds.
		With this knowledge, we are able to define a diffeological retraction.

			\begin{definition}[Diffeological retraction]
				\label{def:Retraction}
				Let $ X $ be a diffeological space.
				A diffeological retraction of $ X $ is a map $ \mathcal{R} \colon TX \rightarrow X $ such that the following conditions hold:
				\begin{itemize}
					\item[(i)] 
						$ \mathcal{R}_{\mid_{T_xX}} (0) = x$
					\item[(ii)] 
						Let $\xi \in T_xX $ and $ \gamma_{\xi} \colon T_0\RR \to T_xX,\ t \mapsto \mathcal{R}_{\mid_{T_xX}}(t \xi) $. Then $T_0\gamma_{\xi}(0) = \xi$
				\end{itemize}
			\end{definition}


It is worth to mention that if we are able to identify the path derivatives $\d_{\alpha}, \d_{\beta} \in T_xX$ with directions $v,w \in X$, the existence of a path $\gamma$ such that we are able to identify $\d_{\gamma}$ with $v+w$ is not guaranteed.
This can cause difficulties in the second condition of definition~\ref{def:Retraction}.
Therefore, we need to deal with a weaker definition of retractions, the so called \emph{weak diffeological retraction}.

	\begin{definition}
		\label{def:WeakRetraction}
		Let $X$ be a diffeological space and $CX$ the tangent cone bundle of $X$ as defined in (\ref{tangentcone}). 
		A \emph{weak diffeological retraction} of  $ X $ is a map $ \mathcal{R} \colon CX \rightarrow X $ such that the following conditions hold:
		\begin{itemize}
			\item[(i)] 
				$ \mathcal{R}_{\mid_{C_xX}} (0) = x$
			\item[(ii)] 
				Let $\xi \in C_xX $ and $ \gamma_{\xi} \colon C_0\RR \to C_xX,\ t \mapsto \mathcal{R}_{\mid_{C_xX}}(t \xi) $. 
				Then, $C_0\gamma_{\xi} = \xi$.
		\end{itemize}
		In (ii), the map $C_0\gamma_{\xi}$ is a tangential cone map and given by
		\begin{align*}
			C_0\gamma_{\xi} \colon C_0\RR \to C_0X, \d_{\alpha} \mapsto \d_{\gamma_{\xi} \circ \alpha}.
		\end{align*}
	\end{definition}

Now, we are ready to formulate the gradient descent method on diffeological spaces. In the next section, the algorithm is formulated and demonstrated on an example.


\section{Formulation and application of diffeological optimization algorithms}
	\label{sec:Star}

	In this section, we first generalize algorithm~\ref{Algo} to diffeological spaces. Afterwards, we apply the formulated algorithms to an example to find the minimum of a function in a diffeological space.

	We start with formulating the steepest descent method on diffeological spaces in algorithm~\ref{AlgoDiff} (first, without specifying the step size strategy). For this algorithm, we need the objects defined in the sections above; in particular, the diffeological gradient and retraction. Please note that we are using the diffeological retraction instead of the exponential map in algorithm~\ref{AlgoDiff}. This differs from algorithm~\ref{Algo}, in which the exponential map is used, but it is also possible to use a retraction mapping in algorithm~\ref{AlgoDiff} as mentioned in remark~\ref{rem:Algo}.

	\begin{algorithm}
			\caption{Steepest descent method on the diffeological space $X$}
					\label{AlgoDiff}
			\begin{algorithmic}
				\State \textbf{Require:} Objective function $f$ on a diffeological Riemannian space $X$;
				\\\phantom{\textbf{Require:} }diffeological retraction $\mathcal{R}$ on $X$; 
				step size strategy.
				\vspace{.1cm}
				\State \textbf{Goal:} Find the solution of $\min\limits_{x\in X}f(x)$.
				\vspace{.1cm}
				\State \textbf{Input:} Initial data $x_0\in X$. 
				\vspace{.3cm}
				
				\State \textbf{for} $k=0,1,\dots$ \textbf{do}
				\vspace{.1cm}
				\State [1] Compute $\text{grad}f(x_k)$ denoting the diffeological shape gradient of $f$ in $x_k$.
				\vspace{.1cm}
				\State [2] Compute step size $t_k$.
				\vspace{.1cm}
				\State [3] Set $			x_{k+1}:= \mathcal{R}_{\mid_{T_{x_k}X}}\left(-t_k \text{grad}f(x_k)\right).$
				\vspace{.1cm}
				\State \textbf{end for}
				\vspace{.3cm}
			\end{algorithmic}
	\end{algorithm}

	\begin{remark}
		In practice, it is also possible to work with a weak diffeological retraction instead of a retraction algorithm~\ref{AlgoDiff}.
			However, one needs to take into account that a weak diffeological retraction is not strong enough if the gradient of the objective function is not an element of the tangent cone.
	\end{remark}

Now, we apply algorithm \ref{AlgoDiff} to an example. We first introduce the diffeological space in which we want to solve our optimization problem.
	
	\begin{definition}[The star]
		\label{def:TheStar}
		Let $v_i\in\RR^2$ for $i\in\NN$ such that $v_i \neq v_j$ for $i \neq j$ and $\| v_i \| = 1 $ for all $i \in \NN$, where $\|\cdot \|$ denotes the Euclidean norm.
		The set
		\begin{align*}
			\mathcal{S} = S(\{\left.v_i\right| i\in\NN\}):=\left\{\left.\mathbb{R} \cdot v_{i}\right| i \in \NN \right\} \subset \RR^2
		\end{align*}
		is called the star.
		We define the natural inclusions from $\RR$ to $S$ by
		\begin{align*}
			\iota_{v_i} \colon \RR \to S,\ r \mapsto rv_i
		\end{align*}
		for $i \in \NN$ as well as the shift to $x \in S$ by 
		\begin{align*}
			\widetilde{\iota_{v_i}} \colon \RR \to S,\ r \mapsto rv_i + x.
		\end{align*}
	\end{definition}
	
	Figure~\ref{fig:4lines} illustrates the set $S(\{v_1,v_2,v_3,v_4\})$ for $v_1 = (1,0)^{\top}$, $v_2 = (0,1)^{\top}$, $v_3 = \tfrac{1}{\sqrt{2}}(1,1)^{\top}$ and $v_4 = \tfrac{1}{\sqrt{2}}(-1,1)^{\top}$ in order to get an impression of the star $\mathcal{S}$.
	
	\begin{figure}[htbp]
		\centering
		\includegraphics[width=.5\textwidth]{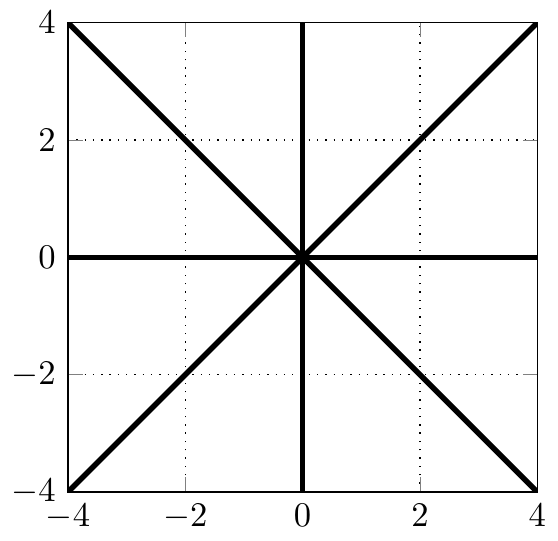}
		\caption{Visualization of $S(\{v_1,v_2,v_3,v_4\})$ in $[-4,4] \times [-4,4]$, where $v_1 = (1,0)^{\top}$, $v_2 = (0,1)^{\top}$, $v_3 = \tfrac{1}{\sqrt{2}}(1,1)^{\top}$ and $v_4 = \tfrac{1}{\sqrt{2}}(-1,1)^{\top}$}
		\label{fig:4lines}
	\end{figure}

		Next, we endow $\mathcal{S}$ with a diffeology.
		There are several choices for a diffeology on the star $\mathcal{S}$.
		We consider $\mathcal{S}$ as a subset of $\RR^2$ and, therefore, we equip $\mathcal{S}$ with the subset diffeology of $\RR^2$ and denote it by $D_{\mathcal{S}}$.
		We notice that $D_{\mathcal{S}}$ is generated by the set of natural inclusions $\{\iota_{v_i} \mid i \in \NN\}$.
		
		We want to solve a minimization problem in the diffeological space $(\mathcal{S}, D_{\mathcal{S}})$. Thus, we need to discuss
		 the tangent cone and the tangent space and define a Riemannian structure on $\mathcal{S}$. We do this before formulating the problem.

		Let $\alpha$ be a path in $\mathcal{S}$ through $x \in \mathcal{S}$.
		Then, there are a $k \in \NN$ and an $r \in \NN$ such that
		\begin{align*}
			\d_{\alpha} = r \cdot \d_{\widetilde{\iota_{v_k}}}.
		\end{align*}
		For the path derivative, it is only important how $\alpha$ behaves near zero and how the map $T_0\alpha$ looks like.
		Obviously, $T_0\alpha = r \cdot T_0 \widetilde{\iota_{v_k}}$ for some $r \in \RR$ and $k \in \NN$ such that $\alpha(0) \in \RR \cdot v_k$.
		Thus, we obtain the following theorem with definition~\ref{def:CxX} and definition~\ref{def:TxX}.
		\begin{lemma}
		\label{le:star_ts}
			The tangent cone of $\mathcal{S}$ at $x \in \mathcal{S}$ is given by
			\begin{align*}
				C_x\mathcal{S} = \{ r \cdot \d_{\widetilde{\iota_{v_i}}} \mid r \in \RR,\ i \in \NN \text{ such that } x \in \RR \cdot v_i\}.
			\end{align*}
			Moreover, 
			\begin{align*}
				T_x\mathcal{S} =\mathrm{span} \{ r \cdot \d_{\widetilde{\iota_{v_i}}} \mid r \in \RR,\ i \in \NN \text{ such that } x \in \RR \cdot v_i \}
			\end{align*}
			is the tangent space of $\mathcal{S}$ at $x \in \mathcal{S}$.
		\end{lemma}

		\begin{remark}
			We need to distinguish between two cases, $x \neq 0$ and $x = 0$, to detail the tangent cone and tangent space defined in lemma~\ref{le:star_ts}.
			
			If $x \neq 0$, then there is only one line through $x$ and, thus,
			\begin{align*}
				C_x\mathcal{S} = \{r \cdot \d_{\widetilde{\iota_{v_k}}} \mid r \in \RR,\, k \in \NN\text{ such that }x \in \RR \cdot v_k \}.
			\end{align*}
			Therefore, the tangent space in $x$ is identical to $C_x\mathcal{S}$, i.e.,
			\begin{align*}
				T_x\mathcal{S} =  \{r  \cdot \d_{\widetilde{\iota_{v_k}}} \mid r \in \RR,\, k \in \NN\text{ such that }x \in \RR \cdot v_k \}.
			\end{align*}

			The case $x = 0$ is more challenging than the case $x \neq 0$.
			Since every line $\RR \cdot v_i$ for $i \in \NN$ goes through zero, the tangent cone through $x$ is given by
			\begin{align*}
				C_0\mathcal{S} = \{ r \cdot \d_{\iota_{v_i}} \mid r \in \RR \text{ and } i \in \NN \}.
			\end{align*}
			In this case, the tangent space does not coincide with the tangent cone. We have
			\begin{align*}
			T_0\mathcal{S} = \mathrm{span} \{C_0\mathcal{S} \}.
			\end{align*}
			More precisely, an element $\xi \in T_0\mathcal{S}$ is given by the finite sum
			\begin{align*}
				\xi = \sum_{i \in I} r_i \cdot \d_{\iota_{v_i}}
			\end{align*}
			for $r_i \in \RR$ and $I \subset \NN$.
			The set $I$ is finite due to the fact that $T_0\mathcal{S}$ is a vector space and every element in a vector space is uniquely given by a finite linear combination of basis elements.
		\end{remark}

		After the discussion about the tangent space of $\mathcal{S}$, we  concentrate now on a Riemannian structure on $\mathcal{S}$.
		We consider the constant map 
		\begin{align*}
			g \colon \mathcal{S} \to \mathrm{Sym}(T\mathcal{S},\RR), x \mapsto \langle \cdot,\ \cdot \rangle_{\mathcal{S}},
		\end{align*}
		where
		\begin{align*}
			\langle \xi,\ \chi \rangle_{\mathcal{S}} = \sum_{i\in I, j \in J} \lambda_i \mu_j \langle v_i,\ v_j \rangle_{\RR^2}
		\end{align*}
		for $\xi =  \sum_{i\in I} \lambda_i \cdot d_{\widetilde{\iota_{v_i}}}$ and $\chi = \sum_{j\in J} \mu_j \cdot \d_{\widetilde{\iota_{v_j}}}$, where $I,J \subset \NN$ are finite subsets.
		A plot of $T_x\mathcal{S}$ is locally of the form
			\begin{align*}
			u \mapsto \sum_{i=1}^{n} \lambda_{i}(u) \d_{\gamma_{i}(u)}
			\end{align*}
			where $\lambda_{i}\colon U \rightarrow \mathbb{R}$ is smooth and $\gamma_{i}\colon U \rightarrow C^{\infty}(\RR, \mathcal{S})$ smooth with $\gamma_{i}(u)(0)=x$ for all $u \in U$ (cf. \cite{vincent}). 
		Thus, $(\xi, \chi) \mapsto \langle \xi,\ \chi \rangle_{\mathcal{S}}$ is smooth.

		In order to be able to apply algorithm~\ref{AlgoDiff}, we still need a concept to update the iterates. This is realized by a retraction.
		We notice that for two generating elements $\d_{\iota_{v_k}}, \d_{\iota_{v_l}}$ in $T_0\mathcal{S}$, it is generally not true that for $\d_{\gamma} \coloneqq d_{\iota_{v_k}} + \d_{\iota_{v_l}}$ there exists a $v_j \in (v_i)_{i \in \NN}$ such that $\gamma$ corresponds to $v_j$. This is because  $ v_k + v_i$ need not to be an element in $\mathcal{S}$.
		For example, let us consider the diffeological space $S(v,w)$ for $v = (1,0)^{\top}, w= (0,1)^{\top}$.
		The tangent cone at $0$ of $S(v,w)$ is then given by 
		\begin{align*}
			C_0S(v,w) =  \{r  \cdot \d_{{\iota_{v}}},r'  \cdot \d_{{\iota_{w}}} \mid r,r' \in \RR\}
		\end{align*}
		analogue to lemma~\ref{le:star_ts}. In contrast, the tangent space at $0$ is given by
		\begin{align*}
			T_0S(v,w) = \{ \lambda \d_{{\iota_{v}}} + \mu \d_{{\iota_{w}}} \mid \lambda, \mu \in \RR \}.
		\end{align*}
		We could consider $(\d_{{\iota_{v}}} + \d_{{\iota_{w}}})(f)$ for $f \in \Cinf(S(v,w))$
		but $f$ is not defined on the line $\RR \cdot (v+w)$. Thus, we do not have a directional derivative in direction $v+w$, which induces that property (ii) of definition~\ref{def:Retraction} cannot be fulfilled. Therefore, we concentrate on weak diffeological retractions in the following.

		Now, we are ready to solve an optimization problem in the diffeological space $\mathcal{S}$ by applying algorithm~\ref{AlgoDiff}.
		For $s \in \mathcal{S}$ we consider the minimization problem
		\begin{align}
		\label{eq:grad_x}
			\min_{x \in \mathcal{S}} f_s(x)
		\end{align}
		for
		\begin{align}
		\label{objective}
			f_s \colon \mathcal{S} \to \RR,\ x \mapsto 
				\begin{cases}
					\| x-s \|		&	\text{if it exist an }  i \in \NN \text{ s.t. } x,s \in \RR \cdot v_i ,\\
					\| x \| + \|s\|	&	\text{else}.
				\end{cases}
		\end{align}
		
		Next, the diffeological gradient of the objective function (\ref{objective}) and at least a weak diffeological retraction needs to be defined for an application of algorithm~\ref{AlgoDiff} to solve (\ref{eq:grad_x}).
		For the computation of the gradient, we distinguish between two cases: $x = 0$ and $x \neq 0$.
		
		In the  case $x \neq 0$, the tangent space $T_x\mathcal{S}$ is characterized by a vector $v \in \{v_i\mid i \in \NN \}$ for which $x$ is an element in $\RR \cdot v$. 
		Therefore, the diffeological gradient of $f$ at $x = \lambda v_k$ is given by
		\begin{align}
		\label{eq:grad_0}
			\grad f_s(x) = 
			\begin{cases}
				\dfrac{\lambda \d_{{\widetilde{\iota_{v}}}}}{\| x\|}				&	\text{for } s \notin \RR \cdot v ,\\[.3cm]
				\dfrac{(\lambda - \mu) \d_{{\widetilde{\iota_{v}}}}}{\| x-s\|}		&	\text{for } s = \mu v \text{ with }\mu \in \RR.
			\end{cases}
		\end{align}
		
		If $x = 0$ then there exists a vector $v \in \{v_i \mid i \in \NN\}$ for every point $y \in \mathcal{S}$ such that $x,y$ are elements in $\RR \cdot v$.
		In consequence, the gradient of $f$ is given by
		\begin{align*}
			\grad f_s (0) = \dfrac{- \mu \d_{{\iota_{w}}}}{\| s\|} = \d_{{\iota_{w}}}
		\end{align*}
		for $\mu \in \RR$ and $w \in \{v_i \mid i \in \NN\}$ such that $ s = \mu w$.

		We will work with a weak diffeological retraction instead of a diffeological retraction due to the nature of $\mathcal{S}$.
		Since elements in $C_x\mathcal{S}$ are determined by vectors $v \in \{v_i \mid i \in \NN\}$ (cf.~lemma~\ref{le:star_ts}), we consider the map
		\begin{align}
			\label{eq:star_retractrion}
			\mathcal{R} \colon C\mathcal{S} \to \mathcal{S},\ 	\mathcal{R}_{\mid_{C_x\mathcal{S}}} (r \cdot \d_{\widetilde{\iota_{v}}}) = x - r v
		\end{align}

		\begin{lemma}
		\label{le:star_retraction}
			The map \ref{eq:star_retractrion} is a weak diffeological retraction for $\mathcal{S}$.
		\end{lemma}
		
		\begin{proof}
			We consider two cases: $ x \neq 0 $ and $ x = 0 $.
			
			Firstly, let $ x \neq 0$.
			Thanks to lemma~\ref{le:star_ts} we know that $\xi \in C_x\mathcal{S}$ is given by $r \cdot \d_{{\widetilde{\iota_{v}}}}$ for one $v \in \{v_i \mid i \in \NN\}$.
			Therefore, the first property of definition~\ref{def:WeakRetraction} holds obviously. 
			By definition, $\mathcal{R}_{\mid_{C_x\mathcal{S}}}(0) = \mathcal{R}_{\mid_{C_x\mathcal{S}}}(0 \cdot \d_{{\widetilde{\iota_{v}}}}) = x - 0v = x$ holds.
			For checking the second property, we consider $f \in \Cinf (\mathcal{S})$ and  a path $\alpha \colon \RR \to \RR$ with $\alpha(0) = 0$.
			By definition of $C_0\gamma_{\xi}$ we have
			\begin{align*}
				C_0\gamma_{\xi}(\d_{\alpha})[f]	
					&=	d_{\gamma_{\xi} \circ \alpha}[f] 
					=	\dot{\alpha}(0) \d_{\gamma_{\xi}}[f].
			\end{align*}
			With definition~\ref{def:WeakRetraction} we then see that
			\begin{align*}
				  \dot{\alpha}(0) \d_{\gamma_{\xi}}[f]
				= \dot{\alpha}(0) \d_0 \left( f(\mathcal{R}_{\mid_{C_x\mathcal{S}}}(rt \cdot \d_{{\widetilde{ \iota_{v} }}})) \right) 
				= \dot{\alpha}(0) \d_0\left( f(x-trv) \right)
			\end{align*}
			holds.
			Thanks to definition~\ref{def:TheStar} we obtain
			\begin{align*}
				\dot{\alpha}(0) \d_0\left( f(x-trv) \right) = \dot{\alpha}(0) \d_0\left( f(r \widetilde{\iota_{v}}(t)) \right).
			\end{align*}
		Finally, considering definition~\ref{def:PathDerivative} and the definition of $\xi$ gives
			\begin{align*}
						\dot{\alpha}(0) \d_0\left( f(r \widetilde{\iota_{v}}(t)) \right) 
					=	\dot{\alpha}(0) r \cdot \d_{{\widetilde{\iota_{v}}}}[f] 
					=	\dot{\alpha}(0) \xi[f].
			\end{align*}
			In summary, $C_0\gamma_{\xi}(\d_{\alpha})[f] = \dot{\alpha}(0) \xi[f]$ holds and, thus, property (ii) of definition~\ref{def:WeakRetraction} is fulfilled for $x \neq 0$.
			
			The case $x = 0$ proceeds analogously to the case $x \not= 0$.
		\end{proof}

		We use algorithm~\ref{AlgoDiff} with a fixed step size and the weak diffeological retraction (\ref{eq:star_retractrion}) instead of the exponential map to solve (\ref{eq:grad_x})--(\ref{objective}).
		Let $v = (0,1)^\top$ and $w = \tfrac{1}{\sqrt{2}}(1,1)^\top$ are elements in $\{v_i \mid i \in \NN\}$.
		We consider $s =2\sqrt{2}  (1,1)^\top$, the starting point $x_0 = (0,3)^\top$ and the constant step size $t_k= 1$ for all $k\in\mathbb{N}_0$.
		Thanks to (\ref{eq:grad_x}), (\ref{eq:grad_0}), (\ref{eq:star_retractrion}) and lemma~\ref{le:star_retraction} we obtain the following iterates:
		
			\begin{align*}
		\begin{array}{lclclcl}
		x_0 	&=& (0,3)^{\top}			&\qquad&		x_{4} &=& \tfrac{1}{\sqrt{2}} (1,1)^{\top}	\\ 
		x_1 	&=& (0, 2)^{\top}			&&		x_{5} &=& \sqrt{2} (1,1)^{\top}	\\ 
		x_2 	&=& (0, 1)^{\top}			&&		x_{6} &=& \tfrac{3}{\sqrt{2}} (1,1)^{\top}	\\ 
		x_3 	&=& (0,0)^{\top}		&&		x_{7} &=& 2\sqrt{2} (1,1)^{\top}	
		\end{array}
		\end{align*}

		\noindent
		The iteration steps of the diffeological gradient descent method for our example are visualized in figure~\ref{fig:Star_GV}.
		
		\begin{figure}
			\centering
			\includegraphics[width=.7\textwidth]{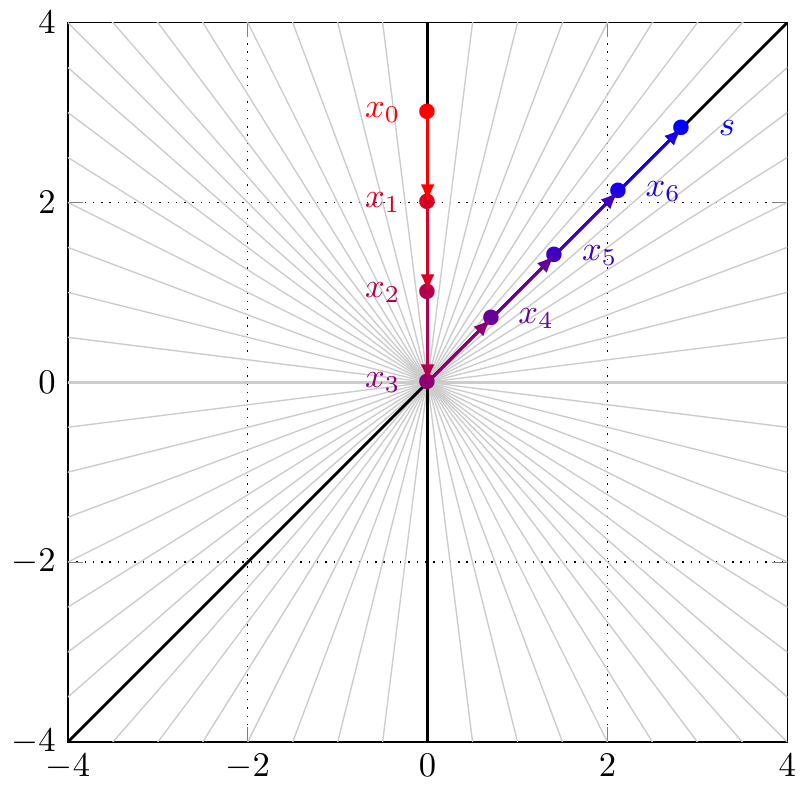}
			\caption{Iteration steps of algorithm~\ref{AlgoDiff}  solving (\ref{eq:grad_x})--(\ref{objective}) for $s = 2\sqrt{2} (1,1)^{\top}$ and starting point $x_0 = (0,3)^{\top}$. In algorithm~\ref{AlgoDiff}, we chose the constant step size $d=1$ and the weak diffeological retraction (\ref{eq:star_retractrion}) instead of the exponential map. Here, $ \mathcal{S}$ is  visualized in $[-4,4] \times [-4,4]$ by only drawing a finite number of lines.}
			\label{fig:Star_GV}
		\end{figure}

	As stopping criterion in the diffeological gradient descent method, we can use $\operatorname{grad}f(x_i) \leq \varepsilon$ as well as $f_s(x_i) \leq \varepsilon$.
	Thanks to the nature of $f_s$ both conditions are valuable.
	In general, a good stopping criterion is not easily to find because diffeological spaces are very different from manifolds and a priori we are not able to take advantage of the information of a Hessian as one do on manifolds.
	Therefore, some further investigations in this direction are necessary and left to future work.
	
	Finally, we discuss the choice of step sizes in the diffeological gradient descent method. In our example above, we choose a constant step size. Unfortunately, a constant step size does not automatically hit the origin which we need to change the direction in the optimization process to solve the model problem. In our example above, we choose an appropriate value for the step size and do not get such a problem. 
	In order to avoid the dependence on suitable choices for the step size values, the step size can be computed exactly by solving an additional minimization problem or, e.g., an Armijo backtracking line search technique can be used to calculate the step size in each iteration.
	In our example, computing the exact step size leads to only two iteration steps.
	Here, only the structure of $\mathcal{S}$ and $f_s$ is responsible for such a fast convergence. 
	The first step is of length $\|x_0\|$ and leads to the origin. The second step is of length $\|s\|$ and gives the optimum. 
	However, as mentioned already above, computing the exact step size requires solving an additional optimization problem which leads in more complex examples to extra computational effort. 
	Further, the structure of the diffeological space complicates  the computation of the solution of an additional optimization problem.

			\begin{algorithm}
				\caption{Steepest descent method on the diffeological space $X$ with Armijo backtracking line search}
					\label{AlgoBacktrakcing}
				\begin{algorithmic}
					\State 
						\textbf{Require:} 
							Objective function $f$ on a diffeological Riemannian space $X$; 
							\\\phantom{\textbf{Require:} }diffeological retraction $\mathcal{R}$ on $X$.
					\vspace{.1cm}
					\State \textbf{Goal:} 
						Find the solution of $\min\limits_{x\in X}f(x)$.
					\vspace{.1cm}
					\State \textbf{Input:} 
						Initial data $x_0 \in X$; 
						constants $\hat{\alpha}>0$ and $	{\displaystyle \sigma, \rho \in (0,1)}$ for Armijo \\\phantom{\textbf{Input: }}backtracking strategy
					\vspace{.3cm}
					
						\State \textbf{for} $k=0,1,\dots$ \textbf{do}
					\vspace{.1cm}
					\State [1] Compute $\text{grad}f(x_k)$ denoting the diffeological shape gradient of $f$ in $x_k$.
					\vspace{.1cm}
					\State [2] Compute Armijo backtracking step size: 
					\vspace{.1cm}
					\State \hspace*{1cm} Set $\alpha := \hat{\alpha}$.
					\State \hspace*{1cm} \textbf{while} $ f\big(\mathcal{R}_{\mid_{T_{x_k}X}}\left(-t_k \text{grad}f(x_k)\right)\big) > f(x^k)-\sigma\alpha \left\|\text{grad}f(x_k)\right\|^2_{T_{x_k}X}$ 
					\State \hspace*{1cm} Set $ \alpha :=\rho \alpha $.
					\State \hspace*{1cm} \textbf{end while}
					\State \hspace*{1cm} Set $t_k:=\alpha$.
					\vspace{.1cm}
					\vspace{.1cm}
					\State [3] Set $			x_{k+1}:= \mathcal{R}_{\mid_{T_{x_k}X}}\left(-t_k \text{grad}f(x_k)\right).$
					\vspace{.1cm}
					\State \textbf{end for}
					\vspace{.3cm}
				\end{algorithmic}
			\end{algorithm}
			
			In practice, the above-mentioned backtracking step size strategy is often used.
			In algorithm~\ref{AlgoBacktrakcing}, the steepest descent method on a diffeological space together with an Armijo backtracking line search strategy is formulated.
			Here, the norm introduced by the Riemannian metric under consideration is needed, $\|\chi\|^2_{T_{x}X}:= g_x( \chi, \chi)$. 
			In our experiment, we apply algorithm~\ref{AlgoBacktrakcing} to solve (\ref{eq:grad_x})-(\ref{objective}).
			Choosing $\sigma = 0.1, \rho = 0.5$ and $\hat{\alpha} = 10$ 
			results in the following iterates:
			\begin{align*}
				\begin{array}{lclclcl}
					x_0 	&=& (0,3)^{\top}			&\qquad&	x_{12} &=& (0.0, 0.0)^{\top}			\\
					x_1 	&=& (0.0, -2.0)^{\top}			&&		x_{13} &=& (3.535534, 3.535534)^{\top}	\\ 
					x_2 	&=& (0.0, 0.5)^{\top}			&&		x_{14} &=& (2.65165, 2.65165)^{\top}	\\ 
					x_3 	&=& (0.0, -0.125)^{\top}		&&		x_{15} &=& (2.872621, 2.872621)^{\top}	\\ 
					x_4 	&=& (0.0, 0.03125)^{\top}		&&		x_{16} &=& (2.817379, 2.817379)^{\top}	\\ 
					x_5 	&=& (0.0, -0.007812)^{\top}		&&		x_{17} &=& (2.831189, 2.831189)^{\top}	\\ 
					x_6 	&=& (0.0, 0.001953)^{\top}		&&		x_{18} &=& (2.827737, 2.827737)^{\top}	\\ 
					x_7 	&=& (0.0, -0.000488)^{\top}		&&		x_{19} &=& (2.8286, 2.8286)^{\top}		\\ 
					x_8 	&=& (0.0, 0.000122)^{\top}		&&		x_{20} &=& (2.828384, 2.828384)^{\top}	\\ 
					x_9 	&=& (0.0, -0.000031)^{\top}		&&		x_{21} &=& (2.828438, 2.828438)^{\top}	\\ 
					x_{10} 	&=& (0.0, 0.000008)^{\top}		&&		x_{22} &=& (2.828424, 2.828424)^{\top}	\\ 
					x_{11} 	&=& (0.0, -0.000002)^{\top}		&&		x_{23} &=& (2.828428, 2.828428)^{\top}
				\end{array}
			\end{align*}
			As one sees, we need 23 iteration steps to the optimum. 
			We compute the iterates with a precision of $10^{-6}$.
			Using the same parameters as above but with higher precision results in more iterations, e.g., for a precision of $10^{-16}$ we need $56$ iterations.
			Of course, these are more iterations than choosing the constant step size above but one needs to keep in mind that we avoid here the dependence of a suitable choice of the step size value, where unsuitable choices lead to a failing algorithm.

\section{Conclusion and outlook}
\label{sec:conclusion}

	In this paper, the initial question 
	\emph{What if we want to optimize on a space that is not a Riemannian manifold?}
	is answered by formulating an 
	optimization technique on diffeological spaces.
	Concentrating on the method of steepest descent on Riemannian manifolds, this paper discusses the necessary objects for a steepest descent method on diffeological spaces.
	We notice that we need a diffeological version of a Riemannian structure, a diffeological gradient as well as an exponential map.
	In order to define a diffeological Riemannian space, the diffeological concept of a tangent space fitting to optimization methods is needed.
	This paper gives a novel definition of a diffeological tangent space based on the  tangent space definition in the Euclidean case that concentrates on paths.
	Moreover, a diffeological Riemannian space is defined.
	With the help of a diffeological Riemannian structure, a diffeological gradient is specified which is the key object of the optimization algorithm.
	Three examples are considered in this paper. First, it is shown that the tangent space of the axis in $\RR^2$  is $\RR^2$. Thus, the Riemannian structure as well as the diffeological gradient are the same as in the Euclidean case.
	The second example covers the set of all polynomials from $ \RR $ to $ \RR $. 
	We prove that the tangent space at any point is given by $ \RR^{\infty} $.
	The third one is considered in the last section of this paper.
	Here, we use our novel definitions to solve an optimization problem with  the steepest descent method in a diffeological space. 
	In particular, we generalize the steepest descent method known from manifolds to diffeological spaces in this paper.
	Moreover, we generalize the Levi-Civita connection and the retraction to diffeological spaces.
	
	The results of this paper leave space for future research.
	For example, the Levi-Civita connection needs to be investigated; in particular, its existence should be clarified.
	Since it is possible to define an exponential map on a smooth manifold using a Levi-Civita connection (cf., e.g., \cite{HelgasonSigurdur}),
	we need to check if we obtain similar results in a diffeological setting.
	It is also left for future work to investigate other ways defining a diffeological exponential mapping and also its first order approximation, the retraction.
	Additionally,
		higher order methods should be investigated on diffeological spaces requiring a diffeological Hessian or at least an approximation of it.

\section*{Acknowledgment}
	This work has been partly supported by the German Research Foundation (DFG) within the priority program SPP1962/2 under contract number WE~6629/1-1.

\bibliographystyle{plain}
\bibliography{CitationGoldammer_New}

\end{document}